\documentclass[11pt, reqno]{amsart}

\usepackage{amsfonts,amssymb,amsmath,amsthm, ulem}
\usepackage[headinclude,DIV13]{typearea}
\usepackage[utf8]{inputenc}
\usepackage{hyperref}
\usepackage{geometry}
\usepackage{graphicx, psfrag,enumerate,enumitem}

\usepackage{soul}
\usepackage{cancel}
\usepackage{color}
 \usepackage{setspace}

\newcommand{\tX}{\widetilde{X}}
\newcommand{\tY}{\widetilde{Y}}

\newcommand{\rtheta}{l}

\newtheorem{thm}{Theorem}[section]
\newtheorem{prop}[thm]{Proposition}

\newtheorem{cor}[thm]{Corollary}
\newtheorem{lemma}[thm]{Lemma}

\newtheorem{ass}{Assumption}

\newtheorem{rem}[thm]{Remark}
\numberwithin{equation}{section}

\newtheorem{innercustomass}{Assumption}
\newenvironment{customass}[1]
  {\renewcommand\theinnercustomass{#1}\innercustomass}
  {\endinnercustomass}
  
\theoremstyle{definition} 
\theoremstyle{definition}

\renewcommand{\P}{\mathbb{P}}
\newcommand{\R}{\mathbb{R}}

\newcommand{\E}{\mathbb{E}}

\newcommand{\N}{\mathbb{N}}

\newcommand{\eps}{\varepsilon}

\newcommand{\brho}{\mathfrak b}

\newcommand{\ud}{\mathrm{d}}

\DeclareMathOperator{\argmin}{argmin}

\numberwithin{equation}{section}

\begin{document}

\title[Parasite infection in a cell population]
{Parasite infection in a cell population with deaths}

\author{Aline Marguet}
\address{Aline Marguet, Univ. Grenoble Alpes, INRIA, 38000 Grenoble, France}
\email{aline.marguet@inria.fr}

\author{Charline Smadi}
\address{Charline Smadi, Univ. Grenoble Alpes, INRAE, LESSEM, 38000 Grenoble, France
 and Univ. Grenoble Alpes, CNRS, Institut Fourier, 38000 Grenoble, France}
\email{charline.smadi@inrae.fr}

\date{}

\maketitle

\begin{abstract}
We introduce a general class of branching Markov processes for the modelling of a parasite infection in a cell population.
Each cell contains a quantity of parasites which evolves as a diffusion with positive jumps. The growth rate, diffusive function and positive jump rate of this 
quantity of parasites depend on its current value. The division rate of the cells also depends on the quantity of parasites they contain. At division, 
a cell gives birth to two daughter cells and shares its parasites between them.
Cells may also die, at a rate which may depend on the quantity of parasites they contain.
We study the long time behaviour of the parasite infection. In particular, we are interested in the quantity of parasites in a `typical' cell and on the survival 
of the cell population. We specifically focus on the influence of two parameters on the probability for the cell population to survive and/or contain 
the parasite infection: the law of the sharing of the parasites between the daughter cells at division and the form of the division and death rates 
of the cells as functions of the quantity of parasites they contain.
 \end{abstract}

 \vspace{0.2in}


\bigskip


 \vspace{0.5cm}


\section*{Introduction}
We introduce a general class of continuous-time and space branching Markov processes for the study of
a parasite infection in a cell population. This framework is general enough to be applied for the modelling of other
structured populations, with individual trait evolving on the set of positive real numbers.
For instance, another application we can think of, similar in spirit, is the modelling of the protein aggregates in a cell population. These latter, usually eliminated by the cells, can undergo sudden increases due to cellular stress (positive jumps), and are known to be distributed unequally between daughter cells (see \cite{rujano2006polarised} for instance).

The dynamics of the quantity of parasites in a cell is given by a Stochastic Differential Equation (SDE) with a diffusive term and positive jumps.
Then, at a random time whose law may depend on the quantity of parasites in the cell, 
this latter dies or divides. At division, it shares its parasites between its two daughter cells.
We are interested in the long time behaviour of the parasite infection in the cell population.
More precisely, we will focus on two aspects of the dynamics: the size of the cell population, and  the amount of parasites in the cells, including the possibility of explosion or extinction of the quantity of parasites in a positive fraction of the cells, or even in every cell. 
We will see that those quantities are very sensitive to the way cell division and death rates depend on the quantity of parasites in the cell, 
and to the law of the sharing of the parasites between the two daughter cells at division.

In discrete time, from the pioneer model of Kimmel \cite{Kimmel}, many studies have been conducted on branching within branching processes to study the host-parasite dynamics: on the associated quasistationary distributions \cite{bansaye2008}, considering random environment and immigration events \cite{bansaye2009}, on multitype branching processes \cite{alsmeyer2013, alsmeyer2016}. In continuous time, host-parasites dynamics have also  been studied using two-level branching processes, in which the dynamics of the parasites is modelled by a birth-death process with interactions \cite{meleard2013evolutive, osorio2020level} or a Feller process \cite{BT11,BPS}.

Some experiments, conducted in the TAMARA laboratory, have shown that cells distribute unequally their parasites between 
their two daughter cells \cite{stewart2005aging}.
This could be a mechanism aiming at concentrating the parasites in some cell lines in order to `save' the remaining lines.
It is thus important to understand the effect of this unequal sharing on the long time behaviour of the infection in the cell population.
This question has been addressed by Bansaye and Tran in \cite{BT11}. They introduced and studied branching Feller diffusions with a cell division rate depending on the 
quantity of parasites in the cell and a sharing of parasites at division between the two daughter cells according to a random
variable with any symmetric distribution on $[0,1]$. They provided some extinction criteria for the 
infection in a cell line, in the case where 
the cell division rate is constant or a monotone function of the quantity of parasites in the cell, as well as recovery criteria
at the population level, 
in the constant division rate case. In \cite{BPS}, Bansaye and coauthors extended this study by providing the long
time asymptotic of the recovery rate in the latter case.
Our work further extends these results in several directions. First, we allow the parasites' growth rate and diffusion coefficient in a cell to depend on the quantity of 
parasites it contains.
Second, we add the possibility to have positive jumps in the parasites dynamics, with a rate which may depend on the current
quantity of parasites in the cell.
Third, we allow the cell division rate to depend non monotonically on the quantity of parasites.
This situation is more difficult to study than the previous ones, as the genealogical tree of the cell population depends on 
the whole history of the quantity of parasites in the different cell lines.
Finally, we add the possibility for the cells to die at a rate which may depend on the quantity of parasites they contain. 
To our knowledge, this is the first work that takes into account possible deaths in the population for the study of the 
proliferation of parasites in a cellular population.

For the study of structured branching populations, a classical method to obtain information on the distribution of a trait in the population is to introduce a spinal decomposition. It consists in 
distinguishing a particular line of descent in the population, constructed from a size-biased tree \cite{LPP}, and to prove that the dynamics of the trait along this 
particular lineage is representative of the dynamics of the trait of a typical individual in the population, {\it i.e.} an individual picked uniformly at random. The link between the spine and the population process is given by Many-to-One formulas (see \cite{bansaye2011limit, marguet2016uniform} and references therein). Moreover, we refer to \cite{georgii2003supercritical,hardy2009spine,bansaye2011limit,cloez2017limit,marguet2016uniform, marguet2017law} for general 
results on these topics in the continuous-time case.

Our proof strategy consists in introducing such a spinal process, also known as auxiliary process.
Then, we investigate the long time behaviour of this auxiliary process that corresponds to the trait of a uniformly sampled individual in the population, 
and deduce properties on the 
long time behaviour of the process at the population level, extending previous results derived for a smaller class of 
structured Markov branching processes (see \cite{BT11,bansaye2011limit,cloez2017limit} for instance). In the case of a constant growth rate for the cellular population, the auxiliary process belongs to the class of continuous-state non-linear branching processes that has been studied in \cite{li2017general, companion}. In the general case, it is time-inhomogeneous and some adaptations of existing results on this class of processes are required. Note that the idea of studying a specific line of descent has also been used in discrete time studies of host-parasites dynamics \cite{bansaye2008,alsmeyer2016}.

The paper is structured as follows. In Section \ref{section_model}, we define the population process and give assumptions ensuring its existence and uniqueness 
as the strong solution to a SDE. In Section \ref{sec_ct_b_and_d}, we consider the case of constant division and death rates. 
Section \ref{sec_mean_numb_cells} is dedicated to the study of the asymptotic behaviour of the mean number of cells alive in the population for various 
dynamics for the parasites. We also compare different strategies for the sharing of the parasites at division and give explicit 
conditions ensuring extinction or survival of the cell population. In Section \ref{section_beta}, we focus on 
the case of a parasites dynamics without stable positive jumps and study the asymptotic behaviour of the proportion of infected cells. 
Similar questions are investigated in Section \ref{sec:LDCD} in the case of a linear division rate and a constant death rate. 
In Section \ref{sec_constant_diff} we provide necessary conditions for the cell population to contain the infection or for the quantity of parasites 
to explode in all the cells.
Sections \ref{sec:MTO} and \ref{sec:proofs} are dedicated to the proofs.
\\

In the sequel $\N:=\{0,1,2,...\}$ will denote the set of nonnegative integers, $\R_+:=[0,\infty)$ the real line,  $\bar{\R}_+:=\R_+\cup \{ + \infty \}$,
and $\R_+^*:=(0,\infty)$.
We will denote by $\mathcal{C}_b^2(A)$ the set of twice continuously differentiable bounded functions on a set $A$. Finally, for any stochastic 
process $X$ on $\bar{\mathbb{R}}_+$ or $Z$ on the set of point measures on $\bar{\mathbb{R}}_+$, we will denote by 
$\mathbb{E}_x\left[f(X_t)\right]=\E\left[f(X_t)\big|X_0 = x\right]$ and $\mathbb{E}_{\delta_x}\left[f(Z_t)\right]=\E\left[f(Z_t)\big|Z_0 = \delta_x\right]$.

\section{Definition of the population process} \label{section_model}

\subsection{Parasites dynamics in a cell}

Each cell contains parasites whose quantity evolves as a diffusion with positive jumps. More precisely, we consider the SDE
\begin{align} \nonumber \label{X_sans_sauts} \mathfrak{X}_t =x + \int_0^t g(\mathfrak{X}_s)ds
+\int_0^t\sqrt{2\sigma^2 (\mathfrak{X}_s)}dB_s &+
\int_0^t\int_0^{p(\mathfrak{X}_{s^-})}\int_{\mathbb{R}_+}z\widetilde{Q}(ds,dx,dz)\\&+
\int_0^t\int_0^{\mathfrak{X}_{s^-}}\int_{\mathbb{R}_+}zR(ds,dx,dz),
\end{align}
where $x$ is nonnegative, $g$, $\sigma \geq 0$ and $p\geq0$ are real functions on $\bar{\mathbb{R}}_+$, $B$ is a standard Brownian motion,  
$\widetilde{Q}$ is a compensated Poisson point measure with intensity 
$ds\otimes dx\otimes \pi(dz)$, $\pi$ is a nonnegative measure on $\mathbb{R}_+$,
$R$ is a Poisson point measure with intensity 
$ds\otimes dx\otimes \rho(dz)$, and $\rho$ is a measure on $\mathbb{R}_+$ with density:
\begin{align*} 
\rho(dz)=\frac{c_\brho\brho(\brho+1)}{\Gamma(1-\brho)}\frac{1}{z^{2+\brho}}\ud z,
\end{align*}
where $\brho \in (-1,0)$ and $c_\brho<0$ (see \cite[Section 1.2.6]{kyprianou2006introductory} for details on stable processes).
Finally, $B$, $Q$ and $R$ are independent.

We will provide later on conditions under which the SDE \eqref{X_sans_sauts} has a unique nonnegative strong solution.
In this case, it is a Markov process with infinitesimal generator $\mathcal{G}$, satisfying for all $f\in C_b^2(\mathbb{R}_+)$,
\begin{align} \label{def_gene}
\mathcal{G}f(x) = g(x)f'(x)+\sigma^2(x)f''(x)&+p(x)\int_{\mathbb{R}_+}\left(f(x+z)-f(x)-zf'(x)\right)\pi(dz)\\&+x\int_{\mathbb{R}_+}\left(f(x+z)-f(x)\right)\rho(dz)\nonumber ,
\end{align}
and $0$ and $+ \infty$ are two absorbing states.
Following \cite{marguet2016uniform}, we denote by $(\Phi(x,s,t),s\leq t)$ the corresponding stochastic flow {\it i.e.} 
the unique strong solution to \eqref{X_sans_sauts} satisfying $\mathfrak{X}_s = x$ and the dynamics of the trait between division events is well-defined.

\subsection{Cell division}

A cell with a quantity of parasites  
$x$ divides at rate $r(x)$ and is replaced by two 
individuals with quantity of parasites at birth given by $\Theta x$ and $(1-\Theta)x$. Here $\Theta$ is a nonnegative random variable on $(0,1)$ 
with associated symmetric distribution $\kappa(d\theta)$ satisfying $\int_0^1 |\ln \theta|\kappa(d\theta)<\infty$.

\subsection{Cell death}

Cells can die because of two mechanisms. 
First they have a death rate $q(x)$ which depends on the quantity of parasites $x$ they carry. We will call it `natural death'. 
The function $q$ may be nondecreasing, because the presence of parasites may kill the cell, or nonincreasing, if parasites slow down the cellular machinery 
(production of proteins, division, etc.).
Second, they can die when the quantity of parasites they carry explodes (i.e. reaches infinity in finite time), as in this case a proper 
functioning of the cell is not possible anymore. Notice that to model this case, we do not `kill the cell' strictly speaking. As infinity is an absorbing state for 
the quantity of parasites in a cell, and as a cell with an infinite quantity of parasites transmits an infinite quantity of parasites to both its daughter cells, 
we let the process evolve and decide that a cell is dead if it contains an infinite quantity of parasites.

\subsection{Existence and uniqueness}

We use the classical Ulam-Harris-Neveu notation to identify each individual. Let us denote by
\begin{equation*} \label{ulam_not} \mathcal{U}:=\bigcup_{n\in\mathbb{N}}\left\{0,1\right\}^{n}
\end{equation*}
the set of possible labels, $\mathcal{M}_P(\bar{\mathbb{R}}_+)$ the set of point measures on $\bar{\mathbb{R}}_+$, and 
$\mathbb{D}(\mathbb{R}_+,\mathcal{M}_P(\bar{\mathbb{R}}_+))$, the set of c\`adl\`ag measure-valued processes. 
For any $Z\in \mathbb{D}(\mathbb{R}_+,\mathcal{M}_P(\bar{\mathbb{R}}_+))$, $t \geq 0$, we write 
\begin{equation} \label{Ztdirac}
Z_t = \sum_{u\in V_t}\delta_{X_t^u},
\end{equation}
where $V_t\subset\mathcal{U}$ denotes the set of individuals alive at time $t$ and $X_t^u$ the trait at time $t$ of the individual $u$. 
Let $E = \mathcal{U}\times(0,1)\times \bar{\mathbb{R}}_+$  and 
$M(ds,du,d\theta,dz)$ be a Poisson point measure on $\mathbb{R}_+\times E$ with intensity 
$ds\times n(du)\times \kappa(d\theta)\times dz$, where $n(du)$ denotes the counting measure on $\mathcal{U}$. 
Let $\left(\Phi^u(x,s,t),u\in\mathcal{U},x\in\bar{\mathbb{\R}}_+, s\leq t\right)$ be a family of independent stochastic flows satisfying \eqref{X_sans_sauts} 
describing the individual-based dynamics.
We assume that $M$ and $\left(\Phi^u,u\in\mathcal{U}\right)$ are independent. We denote by $\mathcal{F}_t$ the filtration generated by the Poisson point measure $M$ and 
the family of stochastic processes $(\Phi^u(x,s, t), u\in\mathcal{U},x\in\bar{\mathbb{R}}_+,s\leq t)$ up to time $t$. \\

We now introduce assumptions to ensure the strong existence and uniqueness of the process. They are weaker than those of previously considered models, and as a consequence, 
we obtain a large class of branching Markov processes for the modelling of parasite infection in a cell population. Points $i)$ to $iii)$ of Assumption {\bf{EU}} 
(for Existence and Uniqueness) ensure that 
the dynamics in a cell line is well defined
(as the unique nonnegative strong solution to the SDE \eqref{X_sans_sauts} up to explosion, and infinite value of the quantity of parasites after explosion); 
points $iv)$ and $v)$ ensure the non-explosion of the cell population size in finite time.

\begin{customass}{\bf{EU}}\label{ass_A}
\begin{enumerate}[label=\roman*)]
\item The functions $r$ and $p$ are locally Lipschitz on $\R_+$, $p$ is non-decreasing and $p(0)= 0$. The function $g$ is continuous on $\R_+$, $g(0)=0$ and for any $n \in \N$ there exists a finite constant $B_n$ such that for any $0 \leq x \leq y \leq n$
\begin{align*} |g(y)-g(x)|
\leq B_n \phi(y-x),
\ 
\text{ where }\ 
\phi(x) = \left\lbrace\begin{array}{ll}
x \left(1-\ln x\right) & \textrm{if } x\leq 1,\\
1 & \textrm{if } x>1.\\
\end{array}\right.
\end{align*}
\item The function $\sigma$ is H\"older continuous with index $1/2$ on compact sets and $\sigma(0)=0$. 
\item The measure $\pi$ satisfies
$$ \int_0^\infty \left(z \wedge z^2\right)\pi(dz)<\infty. $$
\item There exist $r_1,r_2\geq 0$ and $\gamma\geq 0$ such that for all $x\geq 0$
\begin{align*}
r(x)-q(x)\leq r_1x^{\gamma}+r_2.
\end{align*}
\item There exist $c_1,c_2\geq 0$ such that, for all $x\in\mathbb{R}_+$,
\begin{align*}
\lim_{n\rightarrow+\infty}\mathcal{G}h_{n,\gamma}(x)\leq c_1 x^{\gamma}+c_2,
\end{align*}
where $\gamma$ has been defined in iv) and $h_{n,\gamma}\in C_b^2(\mathbb{R}_+)$ is a 
sequence of functions such that $\lim_{n\rightarrow+\infty} h_{n,\gamma}(x)=x^{\gamma}$ for all $x\in\mathbb{R}_+$.
\end{enumerate}
\end{customass}

Recall the definition of $\mathcal{G}$ in \eqref{def_gene}.
Then the structured population process may be defined as the strong solution to a SDE. 

\begin{prop} \label{pro_exi_uni}
Under Assumption \ref{ass_A} there exists a strongly unique $\mathcal{F}_t$-adapted 
c\`adl\`ag process $(Z_t,t\geq 0)$ taking values in $\mathcal{M}_P(\bar{\mathbb{R}}_+)$ such that for all $f\in C_b^2(\bar{\mathbb{R}}_+)$ and $x_0,t\geq 0$, 
\begin{align*}\label{eq:eds-pop}\nonumber
\langle Z_{t},f\rangle = &f\left(x_{0}\right)+\int_{0}^{t}\int_{\mathbb{R}_+}\mathcal{G}f(x)Z_{s}\left(dx\right)ds+M_{t}^f(x_0)&\\\nonumber
 +\int_{0}^{t}\int_{E}&\mathbf{1}_{\left\{u\in V_{s^{-}}\right\}}\left(\mathbf{1}_{\left\{\ z\leq r(X_{s^{-}}^u)\right\}}\left(f\left(\theta X_{s^-}^u \right)+ f\left((1-\theta) X_{s^-}^u \right)-
f\left(X_{s^{-}}^{u}\right)\right)\right.\\
& \left.-\mathbf{1}_{\left\{0<z-r(X_{s^{-}}^u)\leq q(X_{s^{-}}^u)\right\}}
f\left(X_{s^{-}}^{u}\right)\right) M\left(ds,du,d\theta,dz\right),
\end{align*}
where for all $x\geq 0$, $M_{t}^f(x)$ is a $\mathcal{F}_t$-martingale.
\end{prop}

The proof is a combination of \cite[Proposition 1]{palau2018branching} and \cite[Theorem 2.1]{marguet2016uniform} (see Appendix \ref{proof_ex_uni} for details). 
Note that we replaced Condition (1) and (3) of Assumption A in \cite{marguet2016uniform} by Condition {\it iv)} in Assumption \ref{ass_A}. A careful look at 
the proof of \cite[Theorem 2.1]{marguet2016uniform} (in particular (2.5) in \cite[Lemma 2.5]{marguet2016uniform}) shows that in our case, the growth of the 
population is governed by the function $x\mapsto r(x)-q(x)$ so that our condition is sufficient. Note also that in \cite{marguet2016uniform} the exponent 
$\gamma$ of Condition {\it iv)} in Assumption \ref{ass_A} is required to be greater than $1$ but this condition is not necessary for conservative fragmentation 
processes as considered here. 
For the sake of readability we will assume that all the processes under consideration in the sequel satisfy Assumption \ref{ass_A}, but we will not indicate it.\\

We will now investigate the long time behaviour of the infection in the cell population.
As we have explained in the introduction, the strategy to obtain information at the population level is to introduce an auxiliary process providing information on the 
behaviour of a `typical individual'. 
We will provide a general expression for this auxiliary process in Section \ref{sec:MTO}. It involves the mean number of cells in the population, which is not always accessible. 
The computation is however doable in some particular cases, which allows us to study the influence of 
the different parameters on the long time behaviour of the parasite infection (quantity of parasites in the cells, mean number of cells in the population and survival of the cell population), 
in particular the influence of the division strategy (division rate and law for the sharing of the parasites).

\section{Constant birth and death rates} \label{sec_ct_b_and_d}

In this section we assume that $r(\cdot)$ and $q(\cdot)$ are constant functions: the cell division rate and the natural death rate do not depend on the quantity of parasites. 
However, the quantity of parasites in a cell can reach infinity and kill the cell.

\subsection{Mean number of cells alive} \label{sec_mean_numb_cells}

Let us first consider that the quantity of parasites in a cell follows the SDE:
\begin{equation} \label{X_sans_sauts2} \mathfrak{X}_t =x + g\int_0^t \mathfrak{X}_sds+ \int_0^t\sqrt{2\sigma^2\mathfrak{X}_s^2}dB_s
+\int_0^t\int_0^{\mathfrak{X}_{s^-}}\int_{\mathbb{R}_+}zR(ds,dx,dz),
\end{equation}
where $g \geq 0,\sigma \geq 0 $, $x \geq 0$, $B$ is a standard Brownian motion and the Poisson measure $R$ has been defined in \eqref{X_sans_sauts}.
In this simple case, we are able to obtain an equivalent of the number of cells alive at a large time $t$. In particular, we will see how it depends on the way cells divide and share their parasites between their daughter cells.
In order to state the result, let us 
introduce the function 
\begin{equation*}\label{defLaplexp}
\hat{\kappa}(\lambda):= \lambda (g-\sigma^2) + \lambda^2 \sigma^2 + 2 r \left(\E[ \Theta^\lambda]-1\right),
\end{equation*}
for any $\lambda \in (\lambda^-, \infty)$, where 
\begin{equation*} \lambda^-:= \inf \{\lambda <0: \hat{\kappa}(\lambda)<\infty\}.\end{equation*}
The function $\hat{\kappa}$ is the Laplace exponent of a Lévy process (see the proof of Proposition \ref{CNS_ext_ps}), and is thus convex on $(\lambda^-, \infty)$.
Let \label{lambdamoinsm}
$$ \mathbf{m}:=\hat{\kappa}^\prime(0+)=
g-\sigma^2+ 2r\E\left[ \ln\Theta \right]$$ 
and denote by  $\hat{\tau}=\argmin_{(\lambda^-, 0)}\hat{\kappa}(\lambda)$ which is well-defined if $\lambda^-<0<\mathbf{m}$ because $\hat{\kappa}'$ is an 
increasing function.
Finally, we denote by $\mathfrak{C}_t$ the number of cells alive at time $t$.

\begin{prop}\label{CNS_ext_ps}
Assume that the dynamics of the quantity of parasites in a cell follows \eqref{X_sans_sauts2} and that $r(x) \equiv r>0$ and $q(x)\equiv q\geq 0$.
\begin{itemize}
\item[i)] If $\mathbf{m}<0$, then
for every $x>0$ there exists $0<c_1(x)<1$ such that
\begin{equation*}
\underset{t\rightarrow\infty}{\lim}e^{(q-r)t}\E_{\delta_{x}}\left[\mathfrak{C}_t\right]=c_1(x). 
\end{equation*}
\item[ii)] If $\mathbf{m}=0$ and $\lambda^-<0$,
then
for every $x>0$ there exists $c_2(x)>0$ such that
\begin{equation*}
\underset{t\rightarrow\infty}{\lim}\sqrt{t}e^{(q-r)t}\E_{\delta_{x}}\left[\mathfrak{C}_t\right]=c_2(x).  
\end{equation*}
\item[iii)] If $\mathbf{m}>0$, then
for every $x>0$ there exists $c_3(x)>0$ such that
\begin{equation*}
\underset{t\rightarrow\infty}{\lim}t^{\frac{3}{2}} e^{-\hat{\kappa}(\hat{\tau})t} e^{(q-r)t}\E_{\delta_{x}}\left[\mathfrak{C}_t\right]=c_3(x). 
\end{equation*}
\end{itemize}
\end{prop}

Let us focus on the most interesting case when $r>q$. In this case, in absence of parasites, the cell population evolves as a supercritical Galton-Watson process and survives with probability $1-q/r$ (see \cite{athreya1972branching}). In the presence of parasites, in cases $i)$ and $ii)$ the mean number of cells alive at time $t$ goes to infinity with $t$. In case $iii)$ it depends on the value of $\hat{\kappa}(\hat{\tau})$. If
\begin{equation} \label{cond_death} 
\hat{\kappa}(\hat{\tau})+ r -q =\hat{\tau}(g-\sigma^2)+ \hat{\tau}^2 \sigma^2 +  r \left[2\int_0^1 \theta^{\hat{\tau}} \kappa(d\theta) -1\right]-q\leq 0,
\end{equation}
then the mean number of cells alive at time $t$ goes to $0$ when $t$ goes to infinity and thus the cell population is killed by the infection. Otherwise we have the same conclusion that for cases $i)$ and $ii)$.

\begin{cor}\label{Cor_CNS_ext_ps}
Under the assumptions of Proposition \ref{CNS_ext_ps}, for any $x> 0$,
\begin{itemize}
\item[i)] If $q>r$ or if $(\mathbf{m}>0 \text{ and }\hat{\kappa}(\hat{\tau}) + r- q \leq 0), $
$\lim_{t \to \infty} \E_{\delta_{x}}[\mathfrak{C}_t]=0$.
\item[ii)] If $(\mathbf{m}\leq 0\text{ and }r>q)$ or if $(\mathbf{m}>0 \text{ and }\hat{\kappa}(\hat{\tau}) + r- q >0),$
then $ \lim_{t \to \infty} \E_{\delta_{x}}[\mathfrak{C}_t]=\infty$.
\end{itemize}
\end{cor}

Let us now consider a more general dynamics for the quantity of parasites in a cell, namely, the dynamics given by \eqref{X_sans_sauts} with the assumption 
that $g(x)\leq g x$ with $g\geq 0$ for any $x \in \R_+$, and a particular form of the diffusive part.
Notice that in this result, the natural death rate of the cells may depend on the quantity of parasites they contain but has to be bounded.
The diffusive term however may take a rather general form.
Then we may obtain the following sufficient condition for the mean number of cells to go to infinity.
\begin{prop} \label{CS_ext_ps}
Assume that the dynamics of the quantity of parasites in a cell follows \eqref{X_sans_sauts}, with $r(x) \equiv r>q \geq q(x)$, $p(x) = x$, 
$g(x)\leq g x$ and $\sigma(x)^2 = \mathfrak{s}^2(x) x + \sigma^2x^2$ for any $x \in \bar{\R}_+$ with $g,\sigma \in \R_+$.
Assume that the function $\mathfrak{s}$ is H\"older continuous with index $1/2$ on compact sets and that there exists a finite constant $\mathfrak{c}$ such that
for $x \geq 0$, $\mathfrak{s}(x)\sqrt{x} \leq \mathfrak{c} \vee x^\mathfrak{c}$.
If $\mathbf{m}\leq 0$ or $\mathbf{m}>0$ and 
  \begin{equation} \label{cond_pop_infinie} \hat{\kappa}(\hat{\tau})+r-q
= \hat{\tau} (g-\sigma^2)+ \hat{\tau}^2 \sigma^2 +  r \left[2\int_0^1 \theta^{\hat{\tau}} \kappa(d\theta) -1\right]-q >0,
 \end{equation}
then for any $x> 0$
$$\lim_{t \to \infty} \E_{\delta_{x}}[\mathfrak{C}_t] = \infty. $$
\end{prop}

Let us try to understand the effects of $g$, $r$, $q$ and the probability distribution $\kappa$ on cell population survival by studying under which conditions on 
these parameters the inequality \eqref{cond_death} is satisfied.
The effects of $g$, $q$ and $r$ are easy to study. As for any $\lambda \leq 0$,
$$ 2\int_0^1 \theta^{\lambda} \kappa(d\theta) -1 \geq  2\int_0^1  \kappa(d\theta) -1= 1,$$
the higher $r$ is and the more difficult is \eqref{cond_death} to satisfy. If cells divide more often, they are more numerous and they 
get rid of some of their parasites more often.
Hence their quantity is less likely to reach infinity. Similarly, for negative values of $\lambda$, $\hat{\kappa}(\lambda)$ 
is decreasing in $g$, so that if $g$ is bigger, condition \eqref{cond_death} is more likely to hold. This is consistent with the fact that the quantity of parasites 
explodes sooner if it grows faster. Finally, condition \eqref{cond_death} 
is less likely to hold for large values of $q$, which is consistent with the fact that the number of cells alive decreases when the natural death rate $q$ increases.
The effect of $\kappa$, which describes the sharing of the parasites at division, is less intuitive and explicit computations are not always feasible. We nevertheless are able to study some particular cases.\\

For the sake of simplicity, we assume in Sections \ref{sec_unif_law} and \ref{sec_unequal_sharing} that the quantity of parasites in a cell follows the SDE 
\eqref{X_sans_sauts2} with $\sigma=0$. We make this simplification to obtain simple expressions as the focus is here on the role
of $\kappa$ on the mean number of cells alive but general conditions including $\sigma$ could be derived in a similar way.

\subsubsection{Uniform law or equal sharing} \label{sec_unif_law}
If $\kappa(d\theta)=d\theta$ or $\kappa(d\theta)=\delta_{1/2}(d\theta)$, we can explicit the bounds of Corollary \ref{Cor_CNS_ext_ps}.

\begin{cor}\label{Cor_CNS_ext_ps_unif}
Assume that the quantity of parasites in a cell follows the SDE \eqref{X_sans_sauts2} with $\sigma=0$, that $r(x) \equiv r>q(x)\equiv q\geq 0$, and take $x>0$. 
\begin{itemize}
\item[-]If $\kappa(d\theta)=d\theta$,
$$
\begin{array}{llll}
\text{i) }&\text{ if }\quad g\geq  3r-q+2\sqrt{2r(r-q)}\quad &\text{ then }& \quad \lim_{t \to \infty} \E_{\delta_x}[\mathfrak{C}_t] = 0,\\
\text{ii) }&\text{ if }\quad g<  3r-q+2\sqrt{2r(r-q)}\quad &\text{ then }& \quad \lim_{t \to \infty} \E_{\delta_x}[\mathfrak{C}_t] = \infty.\\
\end{array}$$ 
\item[-]If $\kappa(d\theta)=\delta_{1/2}(d\theta)$,
$$\begin{array}{llll}
\text{i) }&\text{ if }\quad g\geq   x_0(r,q)\ln 2\quad &\text{ then }& \quad \lim_{t \to \infty} \E_{\delta_x}[\mathfrak{C}_t] = 0,\\
\text{ii) }&\text{ if }\quad g< x_0(r,q)\ln 2\quad &\text{ then }& \quad \lim_{t \to \infty} \E_{\delta_x}[\mathfrak{C}_t] = \infty,\\
\end{array}$$
where $x_0(r,q)>2r$ is the unique value such that
$$x_0(r,q) =  (r+q)(1+\ln 2r- \ln \left( x_0(r,q)  \right))^{-1}.$$
\end{itemize}
\end{cor}

\begin{rem}
 We can prove that the `uniform sharing' strategy  is always better than the `equal sharing' strategy for the cell population. Indeed 
 by Proposition \ref{CNS_ext_ps}, recalling that $\mathbf{m} = g -2r$ in the first case, and $\mathbf{m}= g  -2r\ln 2$ in the second one,
we obtain
$$
\begin{array}{|c|c|c|}
\hline
g & \multicolumn{2}{c|}{\text{Asymptotic order of magnitude of }\E_{\delta_x}[\mathfrak{C}_t] }\\
\cline{2-3} & \kappa(d\theta) = d\theta & \kappa(d\theta)=\delta_{1/2}(d\theta)\\
\hline
\left[0, 2r\ln 2\right)& e^{(r-q)t} & e^{(r-q)t}\\
\hline
\left\{ 2r\ln 2\right\} & e^{(r-q)t} & t^{-1/2}e^{(r-q)t} \\
\hline
( 2r\ln 2, 2r) & e^{(r-q)t} & t^{-3/2}e^{(\hat\kappa_2(\hat\tau_2) +r-q)t} \\
\hline
\left\{2r\right\} & t^{-1/2}e^{(r-q)t} & t^{-3/2}e^{(\hat\kappa_2(\hat\tau_2) +r-q)t}\\
\hline
( 2r,+\infty) & t^{-3/2}e^{(\hat\kappa_1(\hat\tau_1) +r-q)t} & t^{-3/2}e^{(\hat\kappa_2(\hat\tau_2) +r-q)t}\\
\hline
\end{array}
$$
where for all $g> 2r$, $$\hat\kappa_1(\hat\tau_1) = 2\sqrt{2rg}-g-2r>\hat\kappa_2(\hat\tau_2) 
= \frac{g}{\ln 2} \left(1+\ln 2r- \ln \left( \frac{g}{\ln 2}  \right)\right)-2r.$$
\end{rem}

More generally, we expect that a more unequal strategy is always beneficial for the cell population: `sacrificing' some lineages in order to 
save the other ones'. We 
were not able to prove such a general statement, but we
will try to understand better the effect of unequal sharing in the next section.
 
\subsubsection{Unequal sharing}  \label{sec_unequal_sharing}
We assume that there exists $\theta_0\in (0,1/2)$ such that $$\kappa(d\theta) = \frac{1}{2}\delta_{\theta_0}(d\theta) + \frac{1}{2}\delta_{1-\theta_0}(d\theta).$$ 
In this case, we have $\lambda^-=-\infty$ and for $\lambda \in \R$,
$$\hat{\kappa}(\lambda)=\lambda g + r \left[\theta_0^{\lambda}  + (1-\theta_0)^{\lambda}-2\right],\quad \hat{\kappa}'(\lambda)= g + r  \left[\ln(\theta_0)\theta_0^{\lambda}  + \ln(1-\theta_0)(1-\theta_0)^{\lambda}\right]$$
and
 $$ \mathbf{m}=
g +   r\ln( \theta_0(1-\theta_0) ).$$
Let us focus on the case where $\mathbf{m}>0$, that is to say $g/r>-\ln( \theta_0(1-\theta_0) )$. Then the unique minimum of $\hat{\kappa}$ on $\R$ 
is reached at a point $\hat{\tau}$ characterized by
$$ \frac{g}{r} +  \left[\ln(\theta_0)\theta_0^{\hat{\tau}}  + \ln(1-\theta_0)(1-\theta_0)^{\hat{\tau}}\right]= 0.$$
We notice that $\hat{\tau}$ depends only on $g/r$ and $\theta_0$. 
 Thus the mean number of cells goes to $0$ if the two following conditions are satisfied: 
$$ g/r>-\ln( \theta_0(1-\theta_0) ) \quad \text{and} \quad \hat{\kappa}(\hat{\tau},g,r) +r-q \leq 0, $$
or equivalently
$$ g/r>-\ln( \theta_0(1-\theta_0) )\quad \text{and} \quad \hat{\tau}\frac{g}{r}+\theta_0^{\lambda}  + (1-\theta_0)^{\lambda} -1-q/r \leq 0. $$

\begin{figure}
\center
\includegraphics[scale=0.2]{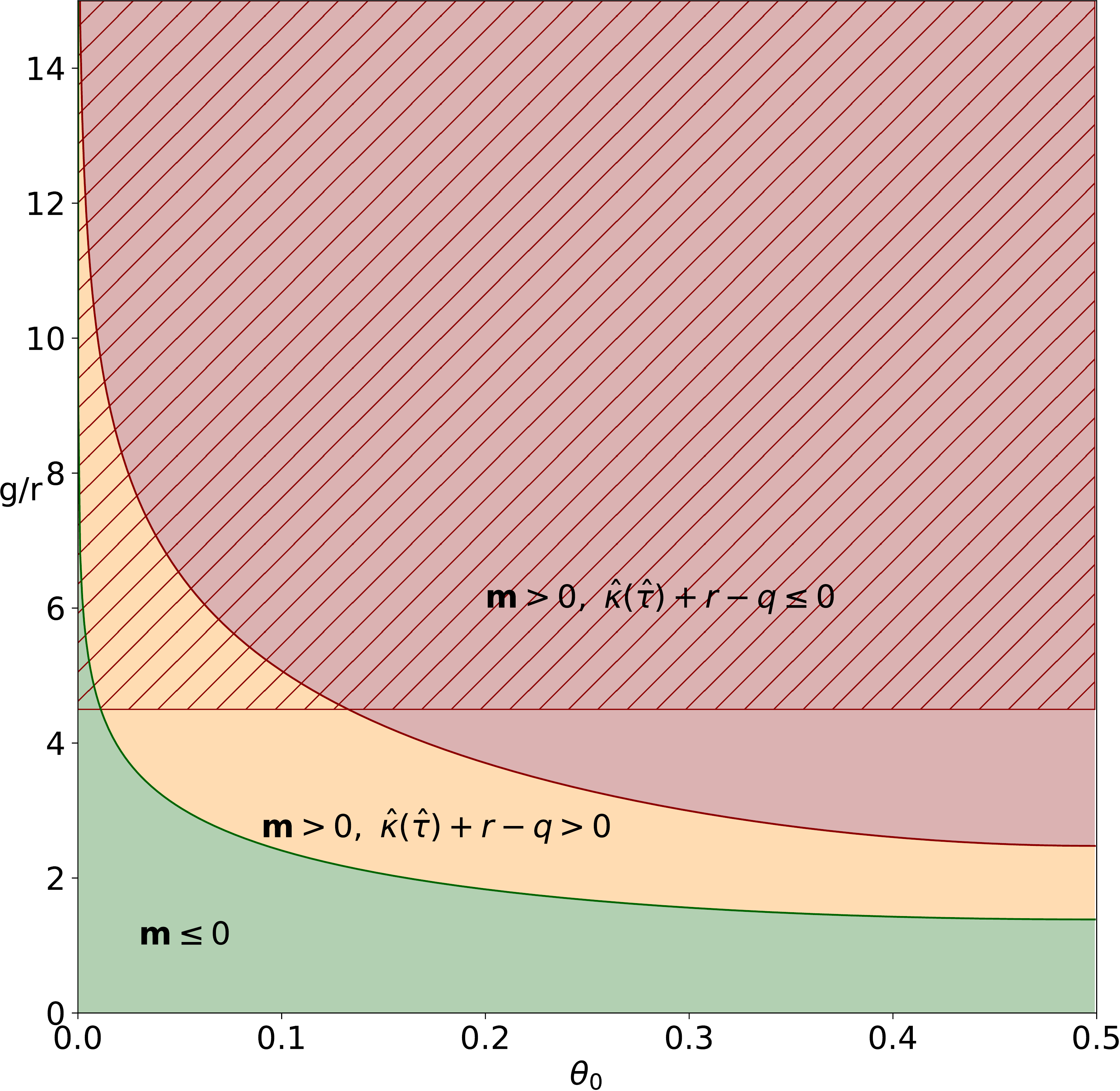}
\caption{Classification of the couples of parameters $(g/r, \theta_0)$ in the case of unequal sharing with $q=r/2$. For a set of 
parameters in the green area, the mean number of cells goes to infinity when time goes to infinity. It is also the case for a set of parameters in the 
orange area but at a smaller rate. For a set of parameters in the red area, the mean number of cells goes to $0$. The hatched area 
corresponds to the 
values of $g/r$ for which the mean number of cells goes to $0$ in the case of a uniform repartition of the parasites at division.}
\label{fig:param_extinction}
\end{figure}
On Figure \ref{fig:param_extinction}, we show the correspondence between the long time behaviour of the mean cell population size and the values of 
$(g/r,\theta_0)$ in the case $q=r/2$. Interestingly, there are values of $g/r$ for which the fate of the cell 
population depends on the strategy of the parasites sharing at division: very asymmetrical divisions (small $\theta_0$) 
can save the cell population. The shapes of the different areas are essentially the same for different values of $q/r$.

\subsection{Quantity of parasites in the cells} \label{section_beta}

We now consider that the dynamics of the parasites in a cell follows the SDE \eqref{X_sans_sauts} without the stable positive jumps, that is to say
\begin{align} \label{X_sans_sauts_stables} \mathfrak{X}_t =x + \int_0^t g(\mathfrak{X}_s)ds
+\int_0^t\sqrt{2\sigma^2 (\mathfrak{X}_s)}dB_s &+
\int_0^t\int_0^{p(\mathfrak{X}_{s^-})}\int_{\mathbb{R}_+}z\widetilde{Q}(ds,dx,dz).
\end{align}

In this case we can observe moderate infections, extinctions of the parasites in the cell population, but also cases where the quantity of parasites goes to infinity with 
an exponential growth in a positive fraction of the cells.\\

In order to state the next result, we need to introduce three assumptions. The first one allows to make couplings and we believe that it can be weakened.

\begin{ass} \label{ass_D}
The measure $\pi$ satisfies $ \int_{\R_+} z\pi(dz)<\infty.$
\end{ass}
Note that the weaker condition $ \int_{\R_+} \ln(1+z)\pi(dz)<\infty $, which is required in \cite{companion}, is therefore satisfied under Assumption \ref{ass_D}.
The second assumption provides a condition under which the quantity of parasites may reach the state $0$. It is almost a necessary and sufficient condition (see  \cite[Remark 3.2 and Theorem 3.3]{companion}).

\begin{enumerate}[label=\bf{(LN0)}]
\item\label{A2}  There exist $0<a<1,$ $\eta>0$ and $u_0> 0$ such that for all $u\leq u_0$
\begin{equation*}
\frac{g(u)}{u}-a\frac{\sigma^2(u)}{u^2} \leq -\ln(u^{-1}) \left(\ln\ln(u^{-1})\right)^{1+\eta}.
\end{equation*}
\end{enumerate}
The third assumption ensures that the process does not explode in finite time almost surely (see \cite[Theorem 4.1]{companion}).
\begin{enumerate}[label=\bf{(SN$\infty$)}]
\item\label{SNinfinity}  There exist $0<a<1,$ and a nonnegative function $f$ on $\mathbb{R}_+$ such that
\begin{equation*}
\frac{g(u)}{u}-a\frac{\sigma^2(u)}{u^2}-r(u)\frac{1-\E\left[\Theta^{1-a}\right]}{1-a}-p(u)I_a(u) =-f(u)+o(\ln u),\quad (u\rightarrow +\infty).
\end{equation*}
\end{enumerate}

Recall that the total number of cells is given by a continuous time birth and death process with individual birth rate $r$ and individual death rate $q$. 
From classical results on branching processes (see for instance \cite{athreya1972branching}) we know that the cell population survives with probability
$0 \vee (1-q/r)$.
The long time behaviour for the quantity of parasites in the cells is described in the next proposition.

\begin{prop} \label{prop:beta_ct:g_var}
Assume that the quantity of parasites in a cell follows the SDE \eqref{X_sans_sauts_stables}, that Assumption \ref{ass_D} holds, and that $r(x) \equiv r>q \equiv q(x)\geq 0$.
\begin{itemize}
 \item[i)] If there exists $\eta>0$ such that for $x \geq 0$,  $$\frac{g(x)}{x}+ 2 r \E[\ln \Theta]> \eta,$$
if the function $x \mapsto (\sigma^2(x)+p(x))/x$ is bounded and if there exists $\eps_1>0$ such that $$\int_{\R_+}z \ln^{1+\eps_1}(1+z)\pi(dz)<\infty,$$
 then for $\eps>0$,
\begin{align*}
\liminf_{t\rightarrow \infty} \mathbb{E}\left[ \mathbf{1}_{\{ N_t\geq 1 \}} \frac{\#\lbrace u\in V_t:X_t^u >e^{(\eta/2r-\eps)t}\rbrace}
{N_t} \right] > 0.
\end{align*} 
\item[ii)] If Assumption \ref{A2} holds and if there exists $\eta>0$ such that for all $x \geq 0$, $$\frac{g(x)}{x}+ 2 r \E[\ln \Theta]<- \eta,$$ then for $\eps>0$ 
$$ \lim_{t\rightarrow \infty}\mathbf{1}_{\{ N_t\geq 1 \}}\frac{\#\lbrace u\in V_t:X_t^u >\eps \rbrace}
{N_t} = 0 \quad \text{in probability}.
 $$
\item [iii)] If Assumptions \ref{A2} and \ref{SNinfinity} hold and if there exist $\eta>0$ and $x_0 \geq 0$ such that for $x \geq x_0$,  
$$\frac{g(x)}{x} - \frac{\sigma^2(x)}{x^2}+ 2 r \E[\ln \Theta]-p(x)\int_0^\infty\left(\frac{z}{x}-\ln\left(1+\frac{z}{x}\right)\right)\pi(dz)<- \eta, $$
then
\begin{equation*}
\lim_{t\rightarrow \infty}\mathbf{1}_{\{ N_t\geq 1 \}}\frac{\#\lbrace u\in V_t: X_t^u >0\rbrace}
{N_t} = 0 \quad a.s.
\end{equation*}
 \end{itemize}
\end{prop}

Proposition \ref{prop:beta_ct:g_var} extends Theorem 4.2 in \cite{BT11} 
allowing for non constant parasite's growth rates, a general class of diffusive functions, positive jumps, as well as the possibility for the cells 
to die at a constant rate.

\section{Linear division rate, constant death rate}\label{sec:LDCD}

In this section, we consider the case of a linear division rate and a constant natural death rate for the cells, and a constant growth for the parasites.
 
\begin{customass}{\bf{LDCG}} \label{ass_E}
There are no stable jumps ($c_\mathfrak{b}=0$),
there exist $\alpha,\beta>0$, $\ g,\ q \geq 0$ such that
$g(x) = gx$, $q(x) \equiv q$, $r(x) = \alpha x+\beta$, $\max(g,\beta)>q$ and
\begin{equation} \label{cond_moment2_pi}
 \int_0^\infty (z \vee z^2)\pi(dz)<\infty.
\end{equation}
\end{customass}

Such a division rate corresponds for the cells to a strategy of linear increase of their 
division rate in order to get rid of the parasites. 
From Lemma \ref{lem_esp_Nt2}, we see that
the quantity $\max(g,\beta)-q$ is the Malthusian growth rate of the population. Therefore, we only consider the case of a growing population.
In this case, we can show that the infection stays moderate, and the population may even recover under some assumptions, under a condition on $p$ and a condition on the  behaviour of $\sigma$ and $p$ at infinity.

\begin{ass} \label{Ass_E'}
 For every $x \geq 0$, $ xp'(x)\geq p(x).$
\end{ass}

\begin{ass} \label{eq:limsup_sigma}
\begin{equation*}
\limsup_{x \to \infty} \frac{\sigma^2(x)}{x^2}<\infty\quad\text{and }\quad \limsup_{x \to \infty} \frac{p(x)}{x^2}<\infty.
\end{equation*}
\end{ass}
Assumption \ref{eq:limsup_sigma} ensures that the quantity of parasites in a typical cell is brought back to small values thanks to division events (see Lemma \ref{lemma:Lyapunov_min}).
We state in the following proposition the possible long time behaviours for the infection.

\begin{prop}\label{prop:temps_long_auxi}
Under Condition \ref{ass_E}, suppose that Assumptions \ref{Ass_E'} and \ref{eq:limsup_sigma} hold.
\begin{itemize}
 \item[i)]  For all $\varepsilon>0$ and $x\geq 0$,
\begin{align*}
\lim_{K\rightarrow \infty}\lim_{t\rightarrow \infty} \mathbb{P}_{\delta_x}
\left( \frac{\#\lbrace u\in V_t: X_t^u >K\rbrace} {\E_{\delta_x}[N_t]}>\eps \right) = 0.
\end{align*} 
If moreover $\int_{\R_+}z^6\pi(dz)<\infty$, then 
\begin{align*}
\lim_{K\rightarrow \infty}\lim_{t\rightarrow \infty} \mathbb{P}_{\delta_x}
\left( \mathbf{1}_{\{N_t \geq 1\}} \frac{\#\lbrace u\in V_t: X_t^u >K\rbrace}{N_t}>\eps \right) = 0.
\end{align*}
\item [ii)] If $\int_{\R_+}z^6\pi(dz)<\infty$ and \ref{A2} holds then
\begin{equation*}
\mathbf{1}_{\{N_t \geq 1\}}\frac{\#\lbrace u\in V_t: X_t^u >0\rbrace} {N_t} \to 0 \quad \text{in} \quad \mathbb{L}_2(\delta_{x_0}), \quad (t\rightarrow \infty).
\end{equation*}
 \end{itemize}
\end{prop}

Notice that point {\it ii)} covers the classical diffusive function ($\sigma^2(x)=\sigma^2 x$, $\sigma>0$).
Proposition \ref{prop:temps_long_auxi} extends the results of \cite{BT11} to a class of division rates increasing with 
the quantity of parasites. It is similar in spirit to \cite[Conjecture 5.2]{BT11} in the case of birth rates increasing with 
the quantity of parasites, but Bansaye and Tran considered a case where the division rate is bounded, which is not our case. 
Moreover, we consider positive jumps and various diffusive functions for the growth of the parasites, and add the possibility for the cells to die.

From this result, we see that the proportion of very infected cells goes to $0$ as $t$ tends to infinity so that a linear division rate is sufficient 
to contain the infection, and even recover if the dynamics of the parasites in a cell is such that the probability of absorption of the infection process 
is positive (condition \ref{A2}). Note that the division mechanism allows to contain the spread of the infection, either by a linear division rate 
(with $\alpha>0$ according to Proposition \ref{prop:temps_long_auxi}) and no restriction on the sharing of the parasites at division, 
or by a constant division rate $\beta$, large enough compared to the 
growth of the parasites weighted by the division events $g/(2\E[\ln(1/\Theta)])$ (see {\it ii)} and {\it iii)} of Proposition \ref{prop:beta_ct:g_var}).

\section{Containment or explosion of the infection in more general cases} \label{sec_constant_diff}

In this section, we consider more general cell division and death rates, and we look for sufficient conditions for the quantity of parasites to 
become large (resp. small) in every alive cell. The idea is to find a function characterizing the parasite growth rate in a typical cell and to 
compare it to the growth rate of the population size.

Let us be more precise on the assumptions entailing these long time behaviours. For $a \in \R_+ \setminus \{1\}$, we introduce when it is well-defined the function $G_a$, for $x>0$ via
\begin{align*}
G_a(x) : =& (a-1)\frac{g(x)}{x}-a(a-1)\frac{\sigma^2(x)}{x^2}-2r(x)\E[\Theta^{1-a}-1] \nonumber \\
& -x^{-\beta}\int_{\mathbb{R}_+} \left(\left(z+1\right)^{1-a}-1\right)\rho(dz)-p(x)\int_{\mathbb{R}_+} \left(\left(\frac{z}{x}+1\right)^{1-a}-1-(1-a)\frac{z}{x}\right)\pi(dz).
\end{align*}

As we explained in the introduction, we will consider a spinal process giving information on the trait dynamics of a typical individual (see the Section \ref{section_constr_aux} for details). As we will see in the proof of Proposition \ref{prop_extin_auxi}, if we denote by $Y$ this spinal process,
then the process $(Y_t^{1-a}e^{\int_0^t G_a(Y_s)ds}, t \geq 0)$ is a local martingale, which entails that, $Y_t^{1-a}$ roughly behaves as 
$e^{-\int_0^tG_a(Y_s)ds}$ and thus, $G_a$ contains informations on the dynamics of the quantity of parasites in a typical cell. 
Moreover, the growth rate of a cell population with a constant quantity $x$ of parasites is $r(x)-q(x)$. The next assumptions combine conditions on 
those two key quantities, leading to results on the asymptotic behaviour of the infection in the entire population.
\begin{customass}{\bf{EXPL}} \label{ass_expl_tps_fini}
 There exist $a>1$ such that $\E[\Theta^{1-a}]<\infty$ and $0\leq \gamma <\gamma'$ such that 
 $$r(x)-q(x) \leq  \gamma  <\gamma' \leq G_{a}(x), \quad \forall x \geq 0. $$
\end{customass}
\begin{customass}{\bf{EXT}} \label{ass_abs_tps_fini}
There is no stable jumps (that is to say $c_\mathfrak{b}=0$) and there exist $a<1$ and $0\leq \gamma <\gamma'$ such that 
$$r(x)-q(x) \leq  \gamma  <\gamma' \leq G_{a}(x), \quad \forall x \geq 0. $$
\end{customass}
\begin{prop} \label{prop_constant_diff}
Let $K>0$. Then for every $x>0$,
\begin{itemize}
  \item[i)] Under Assumption \ref{ass_expl_tps_fini},
  $$ \lim_{t \to \infty} \P_{\delta_x}\left( \exists u \in V_t, X_t^u \leq K \right) =0. $$
  \item[ii)] Under Assumption \ref{ass_abs_tps_fini},
  $$ \lim_{t \to \infty} \P_{\delta_x}\left( \exists u \in V_t, X_t^u \geq 1/K \right) =0. $$
 \end{itemize}
\end{prop}

Thus in case {\it i)}, the quantity of parasites goes to infinity in all the cells, and
in case {\it ii)}, the quantity of parasites goes to zero in all the cell lines with a probability close to one. \\

The rest of the paper is dedicated to the proofs of the results presented in previous sections. As mentioned before, the proofs 
rely on the construction of an auxiliary process, which gives information on the dynamics of the quantity of parasites in a `typical' cell, 
that is to say a cell chosen uniformly at random among the cells alive. 
But to have information on the long time behaviour of the infection at the population level, we need to derive additional results on the number of cells alive, 
which is not an easy task due to both the death rate and the dependence of the cell division rate in the quantity of parasites.

\section{Many-to-One formula}\label{sec:MTO}

\subsection{Construction of the auxiliary process} \label{section_constr_aux}

Recall from \eqref{Ztdirac} that the population state at time $t$, $Z_t$, can be represented by a sum of Dirac masses. 
 We denote by $(M_t,t\geq 0)$ the first-moment semi-group associated with the population 
process $Z$ given for all measurable functions $f$ and $x,t\geq 0$ by
$$
M_tf(x)=\mathbb{E}_{\delta_x}\left[\sum_{u\in V_t} f(X_t^u)\right].
$$ 
The trait of a typical individual in the population is characterized by the so-called 
auxiliary process $Y$ (see \cite[Theorem 3.1]{marguet2016uniform} for detailed computations and proofs). Its associated time-inhomogeneous 
semi-group is given for $r \leq s \leq t$, $x \geq 0$ by
\begin{equation*} 
P_{r,s}^{(t)}f(x)=\frac{M_{s-r}(fM_{t-s}\mathbf{1})(x)}{M_{t-r}\mathbf{1}(x)},
\end{equation*}
where $\mathbf{1}$ is the constant function on $\mathbb{R}_+$ equal to $1$.
More precisely, if we denote by $m(x,s,t)=M_{t-s}\mathbf{1}(x)$ the mean number of cells in the 
population at time $t$ starting 
from one individual with trait $x$ at time $s$ with $s\leq t$,
then, for all measurable bounded functions $F:\mathbb{D}([0,t],\mathbb{R}_+)\rightarrow\mathbb{R}$, we have:
\begin{equation}\label{eq:mto}
\E_{\delta_{x}}\left[\sum_{u\in V_{t}}F\left(X_{s}^{u},s\leq t\right)\right]=m(x,0,t)\E_{x}\left[F\left(Y_{s}^{(t)},s\leq t\right)\right].
\end{equation}
Here $(Y_{s}^{(t)}, s\leq t)$ is a time-inhomogeneous Markov process whose law is characterized by its associated 
infinitesimal generator 
$(\mathcal{A}_{s}^{(t)}, s\leq t)$ given for $f\in\mathcal{D}(\mathcal{A})$ and $x\geq 0$ by: 
\begin{align*}
\mathcal{A}_{s}^{(t)}f(x)= & \widehat{\mathcal{G}}_{s}^{(t)}f(x)
+2r(x)\int_0^1\left(f\left(\theta x\right)-f\left(x\right)\right)\frac{m(\theta x,s,t)}{m(x,s,t)}\kappa(d\theta),
\end{align*}
where 
\begin{align*}
\mathcal{D}(\mathcal{A})=\left\lbrace f\in\mathcal{C}_b^2(\mathbb{R}_+)\text{ s.t. } m(\cdot,s,t)f\in\mathcal{C}_b^2(\mathbb{R}_+),\ \forall t\geq 0,\ \forall s\leq t \right\rbrace
\end{align*}
and
\label{eq:auxi}
\begin{align*}
\widehat{\mathcal{G}}_{s}^{(t)}f(x)= & \left(g(x)+2\sigma^2(x)\frac{\partial_xm(x,s,t)}{m(x,s,t)}+ p(x)\int_{\mathbb{R}_+}z\left(\frac{m(x+z,s,t)-m(x,s,t)}{m(x,s,t)}\right)\pi(dz)\right)f'(x)\\
& +\sigma^2(x)f''(x) + p(x)\int_{\mathbb{R}_+}(f(x+z)-f(x)-zf'(x))\frac{m(x+z,s,t)}{m(x,s,t)}\pi(dz)\\
& +x\int_{\mathbb{R}_+}(f(x+z)-f(x))\frac{m(x+z,s,t)}{m(x,s,t)}\rho(dz).
\end{align*}

Those formulas come from \cite[Theorem 3.1]{marguet2016uniform}, with $B(x) = r(x)+q(x)$ and $m(x,A) = r(x)(r(x)+q(x))^{-1}\int_0^1\left(\delta_{\theta x}(dy)+\delta_{(1-\theta)x}(dy)\right)\kappa(d\theta)$.
Note that explicit expressions for the mean population size $m(x,s,t)$ are usually out of range.
However, the computations are doable in two particular cases.
In the first one, the difference between the cell division and death rates is constant and in the second one the parasites Malthusian 
growth is constant, the cell division rate is a linear function of the quantity of parasites and the cell death rate is constant, as stated in Assumption \ref{ass_E} (Linear Division Constant Growth).

\subsection{Role of the death rate in the auxiliary process}
In this section, we compare the auxiliary process associated to a population with or without death.
Let $(\tilde{Z}_t,t\geq 0)$ be the previously defined population process to which we add a trait $D_t^u$ to each individual $u$ in the population: 
if $D_t^u=0$, the individual is still alive, if $D_t^u=1$, the individual is dead. To compare the population dynamics with or without death, 
we consider that the trait of the dead individuals still evolves and that they can still divide but their 
descendants will be born with the status $D_t=1$. More precisely, 
$$
\tilde{Z}_t = \sum_{u\in V_t}\delta_{(X_t^u, D_t^u)} = \sum_{u\in V_t^0}\delta_{(X_t^u,0)} + \sum_{u\in V_t^1}\delta_{(X_t^u, 1)},
$$
where $V_t^0$ (respectively $V_t^1$) denotes the alive (respectively dead) individuals in the population at time $t$. 
We denote by $N_t^0$ (respectively $N_t^1$) its cardinal and introduce $\tX_t^u  =(X_t^u, D_t^u)$ for all $u\in \mathcal{U}$ and $t\geq 0$.  Next, we consider the following dynamics:
\begin{itemize}
\item[-] a death event for $u$ leads to set $D_t^u=1$. Therefore, it does not affect dead cells. 
\item[-] a division event does not change the status $D_t^u$ of an individual and its descendants inherit the status of their ancestor. 
\item[-] we extend the generator $\mathcal{G}$ to the functions $f:\bar{\mathbb{R}}_+\times \lbrace 0,1\rbrace \rightarrow \mathbb{R}_+$ such that $f(\cdot, 0), f(\cdot,1)\in\mathcal{C}_b^2(\bar{\mathbb{R}}_+)$. 
\end{itemize}
Then, $(\tilde{Z}_t,t\geq 0)$ is defined as the unique strong solution in $\mathcal{M}_P(\mathbb{R}_+\times\lbrace 0, 1\rbrace)$ to
\begin{multline*}
\langle \tilde{Z}_{t},f\rangle = f\left(x_{0}, 0\right)
+\int_{0}^{t}\int_{\mathbb{R}_+}\mathcal{G}f(\widetilde{x})\tilde{Z}_{s}\left(d\widetilde{x}\right)ds+M_{t}^f(x_0,0)\\
 +\int_{0}^{t}\int_{E}\mathbf{1}_{\left\{u\in V_{s^{-}}\right\}}
\left(  \mathbf{1}_{\left\{z\leq r(X_{s^{-}}^u)\right\}}\left(f\left(\theta X_{s^-}^u, D_s^u \right)+ f\left((1-\theta) X_{s^-}^u, D_s^u \right)-
f\left(X_{s^{-}}^{u}, D_s^u\right)\right)\right.\\
\left.+\mathbf{1}_{\left\{0<z-r(X_{s^{-}}^u)\leq q(X_{s^{-}}^u)\right\}}\left(f\left( X_{s^-}^u, 1 \right)- f\left(X_{s^{-}}^{u}, D_{s^{-}}^u \right)\right)\right)
 M\left(ds,du,d\theta,dz\right),
\end{multline*}
for all $f:\mathbb{R}_+\times \lbrace 0,1\rbrace \rightarrow \mathbb{R}_+$ such that $f(\cdot, 0), f(\cdot,1)\in\mathcal{C}_b^2(\mathbb{R}_+)$, 
where $M_\cdot^f$ is an $\tilde{\mathcal{F}}_t$ martingale ($\tilde{\mathcal{F}}_t$ denotes the canonical extension of $\mathcal{F}_t$).
Let $N_t = N_t^0 + N_t^1$. Introduce 
$$m_0(x,s,t):=\mathbb{E}\left[N^0_t \big| Z_s =\delta_{(x,0)}\right]$$
and consider the auxiliary process $\tY^{(t)}_s = (Y_s^{(t)}, D_s)$ for all $0\leq s\leq t$ and its associated generator 
$\widetilde{\mathcal{A}}_s^{(t)}$ given for all $\psi:\bar{\mathbb{R}}_+\times \lbrace 0,1\rbrace \rightarrow \mathbb{R}_+$ 
such that $\psi(\cdot,0),\psi(\cdot,1)\in\mathcal{D}(\mathcal{A})$, and for all $(x,d)\in\mathbb{R}_+\times \lbrace 0,1\rbrace$, by
\begin{align}\nonumber\label{eq:gene_2var}
\widetilde{\mathcal{A}}_s^{(t)}\psi(x,d) = \widehat{\mathcal{G}}_s^{(t)}\psi(\cdot, d)(x)&+ 2r(x)\int_0^1\left(\psi(\theta x,d)-\psi(x,d)\right)\frac{m((\theta x,d),s,t)}{m((x,d),s,t)}\kappa(d\theta)\\
& + q(x)\left(\psi(x,1)-\psi(x,d)\right).
\end{align}
Using the Many-to-One formula \eqref{eq:mto}, we get
\begin{align*}
m_0(x,s,t) = \mathbb{E}\left[\sum_{u\in V_t}\mathbf{1}_{\lbrace D_t^u = 0\rbrace}\Big| Z_s = \delta_{(x,0)}\right] = m((x,0),s,t)\mathbb{P}\left( D_t = 0\big|
\tY_s^{(t)}= (x,0)\right). 
\end{align*}
As we can see on the expression of the generator of the auxiliary process in \eqref{eq:gene_2var}, $D_s$ switches from $0$ to $1$ at rate $q(x)$ and 
$1$ is an absorbing state. Therefore, 
$$
\mathbb{P}\left( D_t = 0\big|\tY_s^{(t)}= (x,0)\right) = \mathbb{E}\left[\exp\left(-\int_s^t q(Y_u^{(t)})du\right)\big|Y_s^{(t)} = x\right].
$$
Finally,
\begin{align*}
m_0(x,s,t) = \mathbb{E}\left[\exp\left(-\int_s^t q(Y_u^{(t)})du\right)\big|Y_s^{(t)} = x\right]m((x,0),s,t),
\end{align*}
and in the case $q(x)\equiv q\geq 0$ for all $x\geq 0$, we get
$$
m_0(x,s,t) = e^{-q(t-s)}m((x,0),s,t).
$$
In particular, for all $x,y\geq 0$
\begin{align}\label{mxtbis}
\frac{m_0(y,s,t)}{m_0(x,s,t)}=\frac{m((y,0),s,t)}{m((x,0),s,t)}.
\end{align}
The expressions appearing in the generator of the auxiliary process given page \pageref{eq:auxi} are identical with or without death, but the difference might 
be hidden in the ratios of $m(y,s,t)/m(x,s,t)$. In \eqref{mxtbis}, the left-hand (respectively right-hand) side corresponds to the ratio 
appearing in the case of a population process with (respectively without) death. 
From the previous computations, we obtain that
in the case of a constant death rate, the auxiliary 
process $(Y_s^{(t)}, s \leq t)$ is the same as the auxiliary process of a population process without death and 
$$
\mathbb{E}\left[\sum_{u\in V_t^0} f(X_s^u)\Big| Z_r = \delta_x\right] = m_0(x,r,t)\mathbb{E}\left[f\left(Y_s^{u}\right)\big| Y_r^{(t)}=x\right]=e^{-q(t-r)}\mathbb{E}\left[\sum_{u\in V_t} f(X_s^u)\Big| Z_r = \delta_x\right].
$$

\subsection{The case $r(x)-q(x)\equiv r-q $} \label{sec_auxi_ct}

In the case where the cell population growth rate is constant, we have
\begin{align*}
m(x,s,t) = 1 + \int_s^t \mathbb{E}\left[\sum_{u\in V_v} \left(r(X_v^u)-q(X_v^u)\right)\Big|Z_s = \delta_x\right]dv = 1 + (r-q) \int_s^t m(x,s,v)dv.
\end{align*}
Therefore, 
$$
m(x,s,t) = e^{(r-q)(t-s)}. 
$$
In this case, the auxiliary process $Y$ is time-homogeneous and its infinitesimal generator is given for all $x\geq 0$ by
\begin{align*}
\mathcal{A}f(x) = \mathcal{G}f(x) + 2r(x)\int_0^1 (f(\theta x)-f(x))\kappa(d\theta).
\end{align*}

In particular, it can be realised as the unique 
strong solution to the following SDE. For all $t \geq 0$,
\begin{align}\nonumber \label{SDE_Y_diff_ct}
Y_t= x& + \int_0^tg(Y_s)ds+ \int_0^t \sqrt{2 \sigma^2(Y_s) }dB_s+ \int_0^t \int_0^{p(Y_{s^-})}
\int_{\mathbb{R}_+}z\widetilde{Q}(ds,dx,dz)\\
&+\int_0^t\int_0^{Y_{s^-}}\int_{\mathbb{R}_+}zR(ds,dx,dz)+\int_0^t\int_0^{2r(Y_{s^-})} \int_0^1  (\theta-1)Y_{s^-}N(ds,dx,d\theta). 
\end{align}
In other words, the auxiliary process has the same law as the process along a lineage with a cell division rate multiplied by two (see \cite{bansaye2011limit,BT11}).

\subsection{The case $r(x)= \alpha x + \beta $, $q(x)\equiv q$}
\label{sec_constr_lineaire}

Assume that $c_\brho=0$ (no stable positive jumps).
Then under Assumption \ref{ass_E}, a direct computation shows that if $g\neq \beta$, the mean number of individuals can be written
\begin{equation} \label{mxst}
m(x,s,t)  =\frac{\alpha x}{g-\beta}e^{(g-q)(t-s)}+\left(1-\frac{\alpha x}{g-\beta}\right)e^{(\beta-q)(t-s)}.\end{equation}

For the sake of readability, we introduce the following functions for $y>0$, $s,z\geq 0$, and $\theta \in [0,1]$: 
\begin{align}\label{eq:f1}
f_1(y,s):=
gy + \left(2\sigma^2(y)+p(y)\int_{\R_+}z^2\pi(dz)\right)\frac{\alpha \left(e^{(g-\beta)s}-1\right)}{g-\beta+\alpha y\left(e^{(g-\beta)s}-1\right)},
\end{align}
\begin{align}\label{eq:f2}
f_2(y,s,\theta):=
2 (\alpha y+\beta)\frac{g-\beta+\alpha \theta y \left(e^{(g-\beta)s}-1\right)}{g-\beta+\alpha y\left(e^{(g-\beta)s}-1\right)},
\end{align}
and
\begin{align*}
f_3(y,s,z):= 
p(y)\left(1+\frac{\alpha z\left(e^{(g-\beta)s}-1\right)}{(g-\beta)+\alpha y\left(e^{(g-\beta)s}-1\right)}\right).
\end{align*}
We obtain that $\mathcal{A}^{(t)}$ is the infinitesimal generator of the solution to the following SDE, when existence and uniqueness in law
of the solution hold. For $0\leq s \leq t$,
\begin{align}\nonumber\label{eq:EDSauxi}
 Y_s^{(t)} =&  Y_0^{(t)}+\int_0^sf_1(Y_u^{(t)},t-u)du + \int_0^s\int_0^\infty \int_0^{f_3(Y_{u^-}^{(t)},t-u,z)}z\widetilde{Q}(du,dz,dx) \\
& + \int_0^s \sqrt{2\sigma^2\left( Y_u^{(t)}\right)}dB_u+ \int_0^s\int_0^1\int_{0}^{f_2(Y_{u^-}^{(t)},t-u,\theta)}
(\theta-1)Y_{u^-}^{(t)}N(du,d\theta,dz),
\end{align}
where $\widetilde{Q}$, $B$ and $N$ are the same as in \eqref{X_sans_sauts}.\\

The auxiliary process $Y^{(t)}$ can be realised as the unique strong solution to the SDE \eqref{eq:EDSauxi}
under some moment conditions on the measure associated with the positive jumps. We need to consider an additional assumption on $p$ that ensures that the rate of positive jumps $f_3$ of the process $Y$ is increasing with the quantity of parasites, namely Assumption \ref{Ass_E'}.

\begin{prop}\label{prop_sol_SDE_auxi}
 Suppose that Assumptions \ref{ass_E} and \ref{Ass_E'} hold.
 Then, Equation \eqref{eq:EDSauxi} has a pathwise unique nonnegative strong solution.
\end{prop}

The proof of this proposition is given in Appendix \ref{app:prop_sol_SDE_auxi}. The case $g=\beta$ will not be considered in this work, as it entails additional computations and 
does not bring new insights.

Now that we have built auxiliary processes in all cases of interest, and have shown that they can be realised as strong solutions to SDEs, we have all tools in hands to prove the results stated in Sections \ref{sec_ct_b_and_d} to \ref{sec_constant_diff}.
Notice however that we will need to study precisely the convergence properties of the inhomogeneous auxiliary process in the linear division rate case, 
to be able to deduce informations on the quantity of parasites in a typical individual.

\section{Proofs}\label{sec:proofs}

\subsection{Proofs of Section \ref{sec_ct_b_and_d}}
In this section, we derive the results on the behaviour of the population in the case of constant division and death rates.

\begin{proof}[Proof of Proposition \ref{CNS_ext_ps}]
According to Section \ref{sec_auxi_ct}, the auxiliary process is homogeneous in this case, and is the unique strong solution to the following SDE:
\begin{align*} 
Y_t= x + g\int_0^tY_s ds &+ \int_0^t \sqrt{2\sigma^2Y_s^2}dB_s + \int_0^t \int_0^{Y_{s^-}}
\int_{\mathbb{R}_+}zR(ds,dx,dz) \nonumber \\
&+\int_0^t\int_0^{2r} \int_0^1  (\theta-1)Y_{s^-}N(ds,dx,d\theta). 
\end{align*}
We can thus apply \eqref{eq:mto} to the function 
$$F((X_s^u,s \leq t))= \mathbf{1}_{\{X_t^u <\infty\}},$$
and obtain
\begin{equation*}
\E_{\delta_{x}}\left[\mathfrak{C}_t\right]=e^{(r-q)t}\P_{x}\left(Y_{t}<\infty\right),
\end{equation*}
where we recall that $\mathfrak{C}_t$ is the number of cells alive at time $t$ (that is to say containing a finite quantity of parasites).
The study of the asymptotic behaviour of $\E\left[\mathfrak{C}_t\right]$ is thus reduced to the study of the asymptotics of the non-explosion probability of $Y$.
Following \cite{palau2016asymptotic}, the long time behaviour of $\P_{x}\left(Y_{t}<\infty\right)$ depends on the properties of the L\'evy process $L$ given by:
\begin{equation} \label{def_Lt} 
L_t:=(g-\sigma^2)t + \sqrt{2 \sigma^2}B_t + \int_0^t\int_0^{2r} \int_0^1  \ln \theta N(ds,dx,d\theta).
\end{equation}
Its Laplace exponent $\hat{\kappa}$ is 
\begin{equation*}
\hat{\kappa}(\lambda):=\ln \E[e^{\lambda L_1}]
= \lambda (g-\sigma^2) + \lambda^2 \sigma^2 + 2 r \left[\int_0^1 \theta^\lambda \kappa(d\theta) -1\right],
\end{equation*}
for any $\lambda \in (\lambda^-, \infty)$. Recall that $\lambda^-$ and $\mathbf{m}$ have been defined on page \pageref{lambdamoinsm}. Then an application of \cite[Proposition 2.1]{palau2016asymptotic} gives
the three following asymptotics:
\begin{itemize}
\item[i)] If $\mathbf{m}<0$, then
for every $x>0$ there exists $0<c_1(x)<1$ such that
\begin{equation*}
\underset{t\rightarrow\infty}{\lim}\P_{x}(Y_t<\infty)=c_1(x). 
\end{equation*}
\item[ii)] If $\mathbf{m}=0$ and $\lambda^{-}<0$,
then
for every $x>0$ there exists $c_2(x)>0$ such that
\begin{equation*}
\underset{t\rightarrow\infty}{\lim}\sqrt{t}\P_{x}(Y_t<\infty)=c_2(x).  
\end{equation*}
\item[iii)] If $\mathbf{m}>0$, then
for every $x>0$ there exists $c_3(x)>0$ such that
\begin{equation*}
\underset{t\rightarrow\infty}{\lim}t^{\frac{3}{2}} e^{-\hat{\kappa}(\hat{\tau})}\P_{x}(Y_t<\infty)=c_3(x). 
\end{equation*}
\end{itemize}
It ends the proof.
\end{proof}
 
\begin{proof}[Proof of Proposition \ref{CS_ext_ps}]
First, we consider the case $g(x) = gx$ and $q(x)\equiv q$. 
The process $\mathfrak{X}$ solution to \eqref{X_sans_sauts} has the same law as the unique solution to the SDE
\begin{align*} \tilde{\mathfrak{X}}_t =x +& g\int_0^t \tilde{\mathfrak{X}}_sds
+\int_0^t\sqrt{2\sigma^2 \tilde{\mathfrak{X}}^2_s}dB_s+\int_0^t\sqrt{2\mathfrak{s}^2 (\tilde{\mathfrak{X}}_s)\tilde{\mathfrak{X}}_s}dW_s \\+&
\int_0^t\int_0^{\tilde{\mathfrak{X}}_{s^-}}\int_{\mathbb{R}_+}z\widetilde{Q}(ds,dx,dz)+
\int_0^t\int_0^{\tilde{\mathfrak{X}}_{s^-}}\int_{\mathbb{R}_+}zR(ds,dx,dz),
\end{align*}
where $W$ is a Brownian motion independent of $B$, $Q$ and $R$.
Notice that under the assumptions of Proposition \ref{CS_ext_ps}, $y \mapsto \mathfrak{s}(y)\sqrt{y}$ satisfies point $ii)$ of Assumption \ref{ass_A}.
As in the previous case, explicit computations are possible, and if we keep the notation $Y$ for the auxiliary process associated to 
$\tilde{\mathfrak{X}}$ for the sake of readability, we obtain that $Y$ is solution to:
 \begin{align} \label{X_sans_sauts4} Y_t =&x + g \int_0^t Y_sds
+\int_0^t\sqrt{2\sigma^2 Y^2_s}dB_s+\int_0^t\sqrt{2\mathfrak{s}^2 (Y_s)Y_s}dW_s+
\int_0^t\int_0^{Y_{s^-}}\int_{\mathbb{R}_+}z\widetilde{Q}(ds,dx,dz)\nonumber\\&+
\int_0^t\int_0^{Y_{s^-}}\int_{\mathbb{R}_+}zR(ds,dx,dz)+\int_0^t\int_0^{2r} \int_0^1  (\theta-1)Y_{s^-}N(ds,dx,d\theta), 
\end{align}
where $Y_0=x\geq 0$.
Recall the definition of the L\'evy process $L$ in \eqref{def_Lt}. Then by an application of It\^o's formula with jumps 
we can show that for any $x,\lambda,0 \leq s \leq t$,
\begin{align}\label{eq:expZt}
 \exp\left(-Y_se^{-L_s}v_t(s,\lambda, L)\right) = \int_0^s \exp\left(-Y_ue^{-L_u}v_t(u,\lambda, L)\right)
e^{-2L_u}v^2_t(u,\lambda, L)\mathfrak{s}^2(Y_u)Y_udu+ \mathfrak{M}_s,
\end{align}
where $(\mathfrak{M}_s, 0 \leq s \leq t)$ is a local martingale conditionally on $(L_s,0\leq s \leq t)$ and
$v_t(.,\lambda,L)$ is the unique solution to
$$ \partial_s v_t(s,\lambda,L)=e^{L_s}\psi_0\left(e^{-L_s}v_t(s,\lambda,L)\right), \quad v_t(t,\lambda,L)=\lambda, $$
where
$$\psi_0(\lambda) = c_\mathfrak{b} \lambda^{1+\mathfrak{b}} +\int_0^\infty \left( e^{-\lambda z}-1+\lambda z \right)\pi(dz).$$
With our assumptions on the function $\mathfrak{s}$, the process 
$$ \left( \exp\left(-Y_ue^{-L_u}v_t(u,\lambda, L)\right)
e^{-2L_u}v^2_t(u,\lambda, L)\mathfrak{s}^2(Y_u)Y_u,0\leq u \leq t\right) $$
is bounded by a finite quantity depending only on $(L_u,0\leq u \leq t)$ (using that $x\mapsto e^{-x}$ and $x\mapsto e^{-x}x^\mathfrak{c}$ are bounded on $\R_+$).
Hence $(\mathfrak{M}_s, 0 \leq s \leq t)$ is a true martingale conditionally on $(L_s,0\leq s \leq t)$, and from \eqref{eq:expZt} we get
\begin{equation} \label{exp_Lap_CSBPRE}\E_x \left[ e^{-\lambda Y_te^{-L_t}} \right]= \E_x \left[ e^{- Y_te^{-L_t}v_t(t,\lambda,L)} \right]\geq
\E \left[ e^{-xv_t(0,\lambda,L)} \right]. \end{equation}
Using that $\psi_0(\lambda)>c_\mathfrak{b}\lambda^{1+\mathfrak{b}}$, we obtain
\begin{align*}
 \partial_s v_t(s,\lambda,L)\geq c_\mathfrak{b} e^{L_s}\left(e^{-L_s}v_t(s,\lambda,L)\right)^{1+\mathfrak{b}}, \quad v_t(t,\lambda,L)=\lambda,
\end{align*}
which entails
$$ v_t(0,\lambda,L) \leq \left( \lambda^{-\mathfrak{b}}+\mathfrak{b} c_\mathfrak{b} \int_0^t e^{-\mathfrak{b} L_s}ds \right)^{-1/\mathfrak{b}}. $$
Combining this latter with \eqref{exp_Lap_CSBPRE}, we obtain
$$  \E_x \left[ e^{-\lambda Y_te^{-L_t}} \right]\geq  \E \left[ e^{-x\left( \lambda^{-\mathfrak{b}}+\mathfrak{b} c_\mathfrak{b} \int_0^t e^{-\mathfrak{b} L_s}ds 
\right)^{-1/\mathfrak{b}}} \right], $$
and letting $\lambda$ tend to $0$, we finally get:
$$  \P_x \left( Y_t<\infty \right)\geq  \E \left[ e^{-x\left( \mathfrak{b} c_\mathfrak{b} \int_0^t e^{-\mathfrak{b} L_s}ds \right)^{-1/\mathfrak{b}}} \right]. $$
As stated in \cite{palau2016asymptotic}, the right-hand side of the last inequality is equal to the probability of non-explosion before time $t$ for 
a self-similar continuous state branching process in a L\'evy random environment.
Therefore, by  \cite[Proposition 2.1]{palau2016asymptotic}, we get
\begin{equation}\label{mathfrakcx}
\underset{t\rightarrow +\infty}{\liminf}v(\mathbf{m},t)\P_x \left( Y_t<\infty \right) =: \mathfrak{a}(x)>0,
\end{equation}
where
\begin{align*}\left\lbrace
\begin{array}{ll}
v(\mathbf{m},t) = 1,&\quad \text{for }\mathbf{m<0},\\
v(0,t) = \sqrt{t},&\\
v(\mathbf{m},t) = t^{3/2}e^{t\hat{\kappa}(\hat{\tau})},&\quad \text{for }\mathbf{m>0}.\\
\end{array}
\right.
\end{align*}
Next, we consider the auxiliary process $\tilde{Y}$ in the case where the quantity of parasites is described by \eqref{X_sans_sauts}, 
with $p(x)=x$, $\sigma^2(x) = \mathfrak{s}^2(x) x +\sigma^2 x^2$ and $g(x)\leq gx$. In this case $\tilde{Y}$ has the same law as a process satisfying 
\eqref{X_sans_sauts4} replacing $g\int_0^tY_sds$ by $\int_0^tg(Y_s)ds\leq g\int_0^tY_sds$. Hence if we choose this version of $\tilde{Y}$, $\tilde{Y}_t\leq Y_t$ for all $t\geq 0$ using that both SDEs have a unique strong solution and 
that $\tilde{Y_0}=Y_0$. Therefore, 
$$
\mathbb{P}_x(\tilde{Y}_t<\infty)\geq \mathbb{P}_x(Y_t<\infty).
$$
Hence from the Many-to-One formula \eqref{eq:mto} and the assumption that $q(\cdot) \equiv q$, we obtain for any $x >0$ and $t$ large enough:
\begin{align*}
 \E_{\delta_{x}}[\mathfrak{C}_t] = e^{(r-q)t}\mathbb{P}_x(\tilde{Y}_t<\infty)
& \geq  e^{(r-q)t} \mathbb{P}_x(Y_t<\infty)\\
& = e^{(r-q)t} v^{-1}(\mathbf{m},t) \left( v(\mathbf{m},t)\P_x \left( Y_t<\infty \right)\right)\\
& \geq e^{(r-q)t} v^{-1}(\mathbf{m},t) \mathfrak{a}(x)/2,
\end{align*}
where we recall that $\mathfrak{a}(x)$ has been defined in \eqref{mathfrakcx}.
Adding that either $(\mathbf{m}>0\text{ and \eqref{cond_pop_infinie})}$ or $\mathbf{m}\leq 0$ holds under the assumptions of Proposition \ref{CS_ext_ps}, we obtain that 
$$ \lim_{t \to \infty} \E_{\delta_{x}}[\mathfrak{C}_t] = \infty. $$
Now let us come back to the general case where for any $x \geq 0$, $
q(x) \leq q$ for some $q\geq 0$. Then for any $x > 0$ we can couple the process $X$ with 
a process $X^{(q)}$ with death rate $q$ and number of cells alive at time $t$ given by $\mathfrak{C}_t^{(q)}$, and such that
$$ \E_{\delta_{x}}[\mathfrak{C}_t] \geq \E_{\delta_{x}}[\mathfrak{C}^{(q)}_t]. $$
Such a coupling may be obtained for instance by first realizing $X$ and then obtaining $X^{(q)}$ by killing additional cells at rate $q-q(x)$ for a cell 
containing a quantity $x$ of parasites. It ends the proof.
\end{proof}

We now explore how the long time behaviour of the infection depends on the law of the sharing of parasites between the two daughter cells at division. 
We focus in particular on the uniform and the equal sharings, two cases where explicit computations are doable.
 
\begin{proof}[Proof of Corollary \ref{Cor_CNS_ext_ps_unif}]
We first focus on the case $\kappa(d\theta)=d\theta$. We get $\lambda^-=-1$ and for $\lambda>-1$,
$$\hat{\kappa}(\lambda)=\lambda g + 2 r \left[\int_0^1 \theta^\lambda d\theta -1\right]=\lambda g+ 2 r \left[\frac{1}{\lambda+1} -1\right],\quad \hat{\kappa}'(\lambda)= 
g - 2 r \frac{1}{(\lambda +1)^2}, $$
and
 $$ \mathbf{m}=
g+ 2r\int_0^1 \ln\theta d\theta = g- 2r.$$

The minimum of $\hat{\kappa}$ on $(-1,\infty)$ is reached at 
$$ \hat{\tau} = \sqrt{\frac{2r}{g}}-1 $$
and equals
$$ \hat{\kappa}(\hat{\tau})=\left(\sqrt{\frac{2r}{g}}-1\right) g + 2 r \left[\sqrt{\frac{g}{2r}} -1\right] = 2 \sqrt{2rg}-g-2r.  $$
Let us look at the sign of $ \hat{\kappa}(\hat{\tau}) +r-q = 2\sqrt{2rg}-g-r-q.$ This quantity is nonpositive if and only if

$$
8rg\leq g^2+ (r+q)^2+ 2g(r+q).
$$
Therefore, setting $X = g$, we have to solve the second degree polynomial equation
$$
X^2 +2X(q-3r)+(r+q)^2= 0. 
$$
Recall that $r>q$. In this case, the two solutions are given by 
$$X_1 = 3r-q-2\sqrt{2r(r-q)},\ X_2 = 3r-q+2\sqrt{2r(r-q)},$$
so that $\hat{\kappa}(\hat{\tau}) +r-q $ is negative for $g<X_1$ or $g> X_2$.
Notice that $X_1-2r =r-q-2\sqrt{2r(r-q)}=\sqrt{r-q}(\sqrt{r-q}-2\sqrt{2r})<0$ and $X_2>2r$. Then, the condition $(\mathbf{m}>0 \text{ and } \hat{\kappa}(\hat\tau) +r-q\leq 0)$ is equivalent to
$$g\geq 3r-q+2\sqrt{2r(r-q)},
$$
and the first point is proved using Corollary \ref{Cor_CNS_ext_ps} {\it i)}.

For the proof of {\it ii)}, we have $g< 3r-q+2\sqrt{2r(r-q)}$ and we distinguish two cases: if $g\leq 2r$ then $\mathbf{m}\leq 0$ and if 
$ 2r<g< 3r-q+2\sqrt{2r(r-q)}$, then $(\mathbf{m}>0 \text{ and } \hat{\kappa}(\hat\tau) +r-q> 0)$
so that the second point is proved using Corollary \ref{Cor_CNS_ext_ps} {\it ii)}.

 Let us now consider the case where the cells share equally their parasites between their two daughters ($\Theta \equiv 1/2$). 
In this case we have $\lambda^-=-\infty$ and for $\lambda \in \R$,
$$\hat{\kappa}(\lambda)=\lambda g + 2 r \left[2^{-\lambda} -1\right],\quad \hat{\kappa}'(\lambda)= g-  2^{1-\lambda} r \ln 2 $$
and
 $$ \mathbf{m}= g -  2r\ln 2 .$$
 The minimum of $\hat{\kappa}$ on $\R$ is reached at 
$$ \hat{\tau}= \frac{1}{\ln 2} \ln \left( \frac{2r \ln 2}{g} \right) $$
and equals
$$ \hat{\kappa}(\hat{\tau})= g \hat{\tau} + \frac{g}{\ln 2}-2r.$$
Thus to have almost sure extinction of the cell population, the two following conditions must be satisfied: 
\begin{align*}
 2r\ln 2<g \quad \text{and} \quad \frac{g}{\ln 2} \left(1+\ln 2r- \ln \left( \frac{g}{\ln 2}  \right)\right)-r-q \leq 0.
\end{align*}
Let 
$$\varphi(x) = x \left(1+\ln 2r- \ln \left( x  \right)\right)-r-q.$$
We are looking for the sign of $\varphi$ on $[2r,+\infty)$, interval on which the first condition $\mathbf{m}>0$ is satisfied. On this interval, $\varphi$ is decreasing from $r-q>0$ to $-\infty$. Thus, there exists $x_0(r,q)>2r$ such that $\varphi(x_0(r,q))=0$ and
\begin{align*}
\text{if } 2r\ln(2)<g< x_0(r,q)\ln(2),&\quad \text{then } (\mathbf{m}>0\text{ and } \hat{\kappa}(\hat{\tau}) +r-q>0),\\
\text{if }g\geq  x_0(r,q)\ln(2),&\quad \text{then } (\mathbf{m}>0\text{ and } \hat{\kappa}(\hat{\tau}) +r-q\leq 0).
\end{align*}
Finally, applying Corollary \ref{Cor_CNS_ext_ps}, we get
$$
\begin{array}{lll}
\text{if } & g\geq x_0(r,q)\ln(2), & \text{ then } \lim_{t \to \infty} \E[\mathfrak{C}_t]=0,\\
\text{if }2r\ln(2)<&g<x_0(r,q)\ln(2)\text{ or } g\leq 2r\ln 2, & \text{ then }  \lim_{t \to \infty} \E[\mathfrak{C}_t]=\infty,\\
\end{array}
$$
which yields the result.

\end{proof}

We now turn to the proof of the results on the asymptotic behaviour of the quantity of parasites in the cells. Recall that for those results, we consider 
that the dynamics of the parasites in a cell follows \eqref{X_sans_sauts_stables}.
\begin{proof}[Proof of Proposition \ref{prop:beta_ct:g_var}]
From Section \ref{section_constr_aux}, we know that the auxiliary process $Y$ is the unique strong solution to the SDE
\begin{align*}
Y_t= x& + \int_0^tg(Y_s)ds+ \int_0^t \sqrt{2 \sigma^2(Y_s) }dB_s+ \int_0^t \int_0^{p(Y_{s^-})}
\int_{\mathbb{R}_+}z\widetilde{Q}(ds,dx,dz)\\
&+\int_0^t\int_0^{2r} \int_0^1  (\theta-1)Y_{s^-}N(ds,dz,d\theta). \nonumber
\end{align*}

Let us begin with the proof of point {\it ii)}. Note that as $g(x)/x+2r\E[\ln\Theta]<-\eta$ for all $x>0$, \ref{SNinfinity} is satisfied.  From (6.3) of \cite[Theorem 6.2]{companion}, we have
\begin{equation*} \lim_{t\rightarrow +\infty} Y_t = 0 \quad \text{almost surely,} 
\end{equation*}
and combining \eqref{eq:mto} with the fact that $\E_{\delta_x}\left[N_t\right]=e^{(r-q) t}$, we obtain that 
\begin{equation*} 
\E_{\delta_x} \left[ \frac{\sum_{u \in V_t} \mathbf{1}_{\{X_t^u >\eps\}} }{e^{(r-q) t}}\right] \xrightarrow[t\rightarrow \infty]{} 0.
\end{equation*}
Moreover, the fact that $(N_t,t\geq 0)$ is a birth and death process with individual death rate $q$ and individual birth 
rate $r$ also entails that $N_te^{-(r-q) t}$ converges in probability to an exponential random 
variable with parameter $1$ on the event of survival, when $t$ goes to infinity. Hence, we have
\begin{align*}
\mathbf{1}_{\{N_t \geq 1\}}\frac{\sum_{u \in V_t} \mathbf{1}_{\{X_t^{u} >\eps\}} }{N_t} = \frac{\sum_{u \in V_t} \mathbf{1}_{\{X_t^{u} >\eps\}} }{e^{(r-q) t}}\times \frac{\mathbf{1}_{\{N_t \geq 1\}}}{N_te^{-(r-q) t}}\xrightarrow[t\rightarrow \infty]{} 0 \quad \text{in probability}.
\end{align*}
It ends the proof of point {\it ii)}.\\

We now prove point {\it iii)}. Applying again (6.3) of \cite[Theorem 6.2]{companion} to $Y$, we obtain that
$$\mathbb{P}\left( Y_t\neq0\right)\rightarrow 0,\quad (t\rightarrow \infty).$$
From this, similarly as for the proof of point {\it ii)} we obtain that
\begin{equation*} \label{conv_proba1}  \mathbf{1}_{\{N_t\geq 1\}}\frac{\sum_{u \in V_t} \mathbf{1}_{\{X_t^u >0\}} }{N_t} \to 0 \quad \text{in probability}, \quad (t\rightarrow \infty). \end{equation*}

To end the proof of point {\it iii)}, we need to prove that the aforementioned convergence holds almost surely. 
We cannot follow directly the proof of \cite[Theorem 4.2(i)]{BT11} because their Lemma 4.3 concerns Yule processes and does not hold when we take into account the death of cells. However, we can 
prove a result similar to this lemma (see Lemma \ref{lemtechnq} in the Appendix) which is sufficient to get our result.
Except from this lemma the proof is exactly the same and we thus refer to \cite{BT11} for details of the proof.\\

We end with the proof of point {\it i)}.
Applying  \cite[Corollary 6.4.{\it iii)}]{companion} to $Y$, we obtain that
$$ \liminf_{t \to \infty} Y_t e^{-\Lambda_{\rho(t)}} = W, $$
with $\P(W>0)>0$ and where $\Lambda$ is a Lévy process with drift $\bar{\eta}:=\eta/2r$ and $\rho(t)\geq 2rt$.
Writing, for $\eps>0$,
$$  Y_t e^{-(\bar\eta-\eps)t}= Y_te^{-\Lambda_{\rho(t)}}e^{\Lambda_{\rho(t)}-(\bar\eta-\eps)t}, $$
and noticing that $\Lambda_{\rho(t)}-(\bar\eta-\eps)t$ goes to $\infty$ when $t$ goes to $\infty$ (see \cite[Theorem 7.2]{kyprianou2006introductory}), we get
$$ \P_x\left(\liminf_{t \to \infty}Y_t e^{-(\bar\eta-\eps)t}>0 \right)>0, $$
and thus by Fatou's Lemma
$$\liminf_{t \to \infty} \P_x\left(Y_t e^{-(\bar\eta-\eps)t}>0 \right)>0. $$
Hence, using \eqref{eq:mto} we obtain
$$\liminf_{t \to \infty} \E_{\delta_x}\left[ \frac{\sum_{u \in V_t} \mathbf{1}_{\{X_t^u e^{-(\bar\eta-\eps)t}>0 \}} }{e^{(r-q) t}}\right]>0. $$
Now notice that the Cauchy-Schwarz inequality yields
\begin{align*}
 \E^2_{\delta_x}\left[  \frac{\sum_{u \in V_t} \mathbf{1}_{\{X_t^u e^{-(\bar\eta-\eps)t}>0 \}} }{e^{(r-q) t}}\right]&\leq 
 \E_{\delta_x}\left[ \mathbf{1}_{\{N_t \geq 1\}}\left(\frac{\sum_{u \in V_t} \mathbf{1}_{\{X_t^u e^{-(\bar\eta-\eps)t}>0 \}} }{N_t}\right)^2\right]
 \E_{\delta_x}\left[\left(\frac{N_t}{e^{(r-q) t}}\right)^2\right]\\
&  \leq \E_{\delta_x}\left[ \mathbf{1}_{\{N_t \geq 1\}}\frac{\sum_{u \in V_t} \mathbf{1}_{\{X_t^u e^{-(\bar\eta-\eps)t}>0 \}} }{N_t}\right]
 \E_{\delta_x}\left[ \left(\frac{N_t}{e^{(r-q) t}}\right)^2\right],
\end{align*}
where the last inequality comes from the fact that the term in the first expectation in the right-hand side is smaller 
than one. 
The last expectation converges to $C:=1+(r+q)/(r-q)$ as $t$ goes to infinity (see forthcoming Lemma \ref{lem_esp_Nt2} in the case $\alpha=0$). Hence we get
\begin{align*}
0<C^{-1}\liminf_{t \to \infty} \E^2_{\delta_x}\left[ \frac{\sum_{u \in V_t} \mathbf{1}_{\{X_t^u e^{-(\bar\eta-\eps)t}>0 \}} }{e^{(r-q) t}}\right]\leq 
\liminf_{t \to \infty} \E_{\delta_x}\left[\mathbf{1}_{\{N_t \geq 1\}} \frac{\sum_{u \in V_t} \mathbf{1}_{\{X_t^u e^{-(\bar\eta-\eps)t}>0 \}} }{N_t}\right],
\end{align*}
and it ends the proof of point {\it i)}.\end{proof}

\subsection{Proof of Section \ref{sec:LDCD}}

To prove Proposition \ref{prop:temps_long_auxi}, we first need to derive some properties of the auxiliary process $(Y_s^{(t)},s\leq t)$. 
Recall that under the assumptions of Proposition \ref{prop:temps_long_auxi} it is the unique strong solution to \eqref{eq:EDSauxi}.

\subsubsection*{Preliminary results on the auxiliary process}

In what follows, we set $Y_s^{(t)} =Y_t^{(t)}$ for all $s\geq t$. 

The next proposition is an analogue of \cite[Theorem 3.3]{companion} on the absorption of the auxiliary process in finite time and its proof is very similar, except that we have to deal with time dependencies. 
Let
$$
\tau_t^-(0):= \inf\lbrace 0<s\leq t : Y_s^{(t)} = 0\rbrace,
$$
with the convention $\inf\emptyset: = \infty$. Introduce the following assumption: 
\begin{enumerate}[label=\bf{(SN0)}]
\item \label{A1}
There exist $a>1$ such that $\E[\Theta^{1-a}]<\infty$ and a nonnegative function $f$ on $\mathbb{R}_+$ such that
\begin{equation*}
\frac{g(u)}{u}- a\frac{\sigma(u)^2}{u^2}= f(u)+  o( \ln u), \quad u \to 0.
\end{equation*}
\end{enumerate}

This condition ensures that the process is not absorbed at $0$ in finite time.

\begin{prop}\label{prop_extin_auxi}
Suppose that Assumptions \ref{ass_E} and \ref{Ass_E'} hold.
\begin{itemize}
\item[i)] 
If Condition \ref{A1} holds, then $\mathbb{P}_x\left(\tau_t^-(0)<\infty\right)=0$ for all $x>0$.
\item[ii)]
If Condition \ref{A2} holds, then for any $x>0$ and $s>0$, 
$\mathbb{P}_x\left(\tau_t^-(0)<s\right) >0$.
\end{itemize}
\end{prop}

\begin{proof}[Proof of Proposition \ref{prop_extin_auxi}]
 This proof is very similar to the proof of \cite[Theorem 3.3]{companion}. The only modifications are due to the time-inhomogeneity of the auxiliary process, and to the fact that the time interval is restricted to $[0,t]$.
We proceed by coupling to overcome these two difficulties.
First, we prove that \cite[Theorem 3.3]{companion} still holds if the rate of positive jumps 
 depends on time and jump sizes.
Let $t>0$ and consider the process $X$ solution to
\begin{align}\label{eq:EDS}
X_t= X_0+ \int_0^t g(X_s) ds+ \int_0^t\sqrt{2 \sigma^2(X_s) }dB_s&+ \int_0^t\int_{\mathbb{R}_+}\int_0^{p_t(X_{s^-},s,z)}z\widetilde{Q}(ds,dz,dx)\\
&+ \int_0^t  \int_0^{r(X_{s^-})} \int_0^1 (\theta-1)X_{s^-}N(ds,dx,d\theta), \nonumber
\end{align}
where for $x,z \geq 0$ and $s \leq t$,
\begin{equation} \label{bound_f3} p_t(x,s,z) \leq p(x) (1+ \mathfrak{c}_t z), \end{equation}
with $\mathfrak{c}_t$ a finite and positive constant depending on $t$.
Let us define for $a>1$
\begin{align*}
G_a^{(s)}(x) : = & (a-1)\frac{g(x)}{x}-a(a-1)\frac{\sigma^2(x)}{x^2}-r(x)\E[\Theta^{1-a}-1]\\
& -\int_{\mathbb{R}_+}p_t(x,s,z) \left((zx^{-1}+1)^{1-a}-1-(1-a)zx^{-1}\right)\pi(dz).
\end{align*}
Using \eqref{bound_f3} we can show as in the proof of \cite[Remark 3.2]{companion} that under \eqref{cond_moment2_pi}
\begin{align}\label{eq:limsup_p}\nonumber
\limsup_{x\rightarrow 0^+} \left( x^{-2}\int_0^\infty p_t(x,s,z) z^2\left(\int_0^1 (1+zx^{-1}v)^{-1-a}(1-v)dv\right)\pi(dz) \right)\\
\leq C_t \left(\limsup_{x\rightarrow 0^+}p(x)x^{-1}\right)\left(\int_0^\infty (z+z^2)\pi(dz)\right)< \infty
\end{align}
where $C_t$ is a finite constant.
Then, applying Itô's Formula with jumps we can check that \cite[Lemma 7.1 and Equation (7.1)]{companion} still hold with $G_a^{(s)}(X_s)$ instead of $G_a(X_s)$.
The proof of \cite[Theorem 3.3]{companion} is thus unchanged and the results hold also for processes whose rate of positive jumps satisfy 
\eqref{bound_f3}.\\
 
Let us now prove point {\it i)}. Introduce $\tY$ the unique strong solution to
\begin{align*}
\tY_t  = \tY_0 + g\int_0^t \tY_s ds+ \int_0^t \sqrt{2 \sigma^2(\tY_s) }dB_s 
+ \int_0^t \int_{\mathbb{R}_+}\int_0^{f_3(\tY_{s^-},t-s,z)}z\widetilde{Q}(ds,dz,dx)
\\ + \int_0^t \int_0^{2(\alpha \tY_{s^-}+\beta)} \int_0^1  (\theta-1)\tY_{s^-}N(ds,dz,d\theta),
\end{align*}
where the Brownian motion and the Poisson random measures are the same as in \eqref{eq:EDSauxi}, and by convention we 
decide that $f_3(x,u,z)= 0$ if $u \leq 0$. 
Notice that for all $y\geq 0$, $0\leq s \leq t$, $0\leq \theta\leq 1$, and $z\geq 0$, 
$$f_1(y,s)\geq gy, \quad f_2(y,s,\theta)\leq 2(\alpha y+\beta),$$
and $f_3(x,t-s,z)$ is non-decreasing in $x$ (thanks to Assumption \ref{Ass_E'}).
In particular this implies that if $\tY$ is a solution with $\tY_0= Y_0^{(t)}$, then $\tY_s\leq Y_s^{(t)}$ for any $s$ smaller than $t$.
But $\tY$ satisfies assumptions of point {\it i)} of \cite[Theorem 3.3]{companion}, and thus does not reach $0$ in finite time. We deduce that $Y^{(t)}$ 
does not reach $0$ before time $t$.\\

We now prove point {\it ii)}.
First notice that 
for any $x> 0$ and $s \leq t$, the function 
$f_1$ defined in \eqref{eq:f1} satisfies
$$ f_1(x,s) \leq gx + \left(2\sigma^2(x) + p(x)\int_{\R_+}z^2\pi(dz)\right) \frac{A_t}{1+x A_t}=:\bar{g}_t(x), $$
where $A_t = \alpha(e^{(g-\beta)t}-1)/(g-\beta)$.
Let
 $(\bar{Y}_s,s\geq 0)$ be the unique strong solution to
\begin{align}\nonumber\label{eq:tildeY}
\bar{Y}_s = & Y_0^{(t)}+ \int_0^s\bar{g}_t\left(\bar{Y}_u\right)du + \int_0^s\sqrt{2\sigma^2\left(\bar{Y}_u\right)}dB_u + 
\int_0^s\int_0^1\int_{0}^{\bar{r}(\theta)}(\theta-1)\bar{Y}_{u^-} N(du,d\theta,dz)\\
& + \int_0^s \int_0^\infty\int_0^{f_3(\bar{Y}_{s^-},t-s,z)}z\widetilde{Q}(du,dz,dx),
\end{align}
where for all $x\geq 0$, $s\geq 0$ and $\theta\in [0,1]$,
$$
\bar{r}(\theta) := 2\theta\beta\leq f_2(x,s,\theta),
$$
with $f_2$ defined in \eqref{eq:f2},
$B$, $N$ and $\widetilde{Q}$ are the same as in \eqref{eq:EDSauxi} and $f_3(x,u,z)= 0$ if $u \leq 0$.
Then, for all $0\leq s \leq t$, $\bar{Y}_s \geq Y_s^{(t)}.$ As a consequence, if we introduce 
\begin{equation*}
\bar{\tau}^-(0):= \inf\lbrace s\geq 0 : \bar{Y}_s = 0\rbrace,
\end{equation*}
and prove that for any $0<v\leq t$,
\begin{equation} \label{bartau0} \P(\bar{\tau}^-(0)<v)>0, \end{equation}
it will imply that
$$ \P(\tau_t^-(0)<v) \geq \P(\bar{\tau}^-(0)<v)>0  $$
and end the proof.
To prove \eqref{bartau0}, we apply \cite[Theorem 3.3{\it iii)}]{companion} to the process $\bar{Y}$. Notice that here, unlike in \cite[Theorem 3.3]{companion}, the division rate $\bar{r}$ depends on $\theta$. 
However, the dependence in $\theta$ in the division rate can be removed by considering a new Poisson point measure 
$N'$ with a modified fragmentation kernel so that all the results derived above still hold. We refer the reader to Appendix \ref{app:generalization SDE} for more details.
\end{proof}
In the case where the absorption of the auxiliary process occurs with positive probability, 
we are able to prove the convergence of the last part of the auxiliary process trajectory on a time window of any size. 
\begin{prop}\label{prop:conv_auxi}
Let $T\geq 0$. Suppose that Assumptions \ref{ass_E}, \ref{Ass_E'},  \ref{eq:limsup_sigma} hold. 
Then, there exist $C,\overline{c}>0$ 
and a probability measure $\Pi$ on the Borel $\sigma$-field of $\mathbb{D}\left([0,T],\mathcal{X}\right)$ endowed with the Skorokhod distance 
such that for all bounded measurable functions $F:\mathbb{D}\left([0,T],\mathcal{X}\right)\rightarrow \mathbb{R}$ and for all $x\geq 0$,
\begin{align*}
\left|\mathbb{E}\left[F\left(Y_{t+s}^{(t+T)},s\leq T\right)\Big|Y_{0}^{(t+T)} = x \right]- \Pi(F)\right|\leq C e^{-\overline{c}t}\left\Vert F\right\Vert_{\infty}x.
\end{align*}
\end{prop}

We prove the convergence of the auxiliary process by verifying a Foster-Lyapunov inequality and a minoration condition, both stated in Lemma \ref{lemma:Lyapunov_min} below. Those standard conditions were exhibited in \cite{marguet2017law} as an extension of \cite{hairer2011yet} to time-inhomogeneous processes. The Foster-Lyapunov inequality (Condition {\it i)} in Lemma \ref{lemma:Lyapunov_min}) ensures that
$$
\E_x\left[V\left(Y_s^{(t)}\right)\right]\leq e^{-as}V(x)+\frac{d}{a}\left(1-e^{-as}\right),
$$
where $a$ and $d$ are positive constants,
so that the process is brought back to the sublevel sets of $V$. The minoration condition ({\it i.e.} Condition {\it ii)} in Lemma \ref{lemma:Lyapunov_min}) ensures some type of irreducibility of the process on those sublevel sets.
Let 
$$V(x)=x \quad  \text{ for } x\in\mathbb{R}_+.$$

\begin{lemma}\label{lemma:Lyapunov_min}Under the assumptions of Proposition \ref{prop:conv_auxi}, we have the following:
\begin{enumerate}[label=\roman*)]
\item There exist $a,d>0$ such that for all $0\leq s\leq t$ and $x\in\mathbb{R}_+^{*}$,
\begin{align*}
\mathcal{A}_s^{(t)}V(x)\leq -aV(x)+d.
\end{align*}
\item There exists $R>2da^{-1}$ such that for all $r<s\leq  t$, there exist $\alpha_{s-r}>0$ and a probability measure $\nu$ on $\mathbb{R}_+$ 
such that for all Borel sets $A$ of $\mathbb{R}_+$,
\begin{align*}
\inf_{x\leq R}\mathbb{P}\left(Y_{s}^{(t)}\in A\big| Y_r^{(t)}=x\right)\geq \alpha_{s-r} \nu(A).
\end{align*}
\end{enumerate}
\end{lemma}

\begin{proof}{\it{i)}} 
We have
\begin{align*}
\mathcal{A}_s^{(t)}V(x) & = f_1(x,t-s)-\int_0^1f_2(x,t-s,\theta)x(1-\theta)\kappa(d\theta)\\
& \leq gx+ 2\frac{\sigma(x)^2}{x}+\frac{p(x)}{x}\int_{\R_+}z^2\pi(dz) -2\alpha x^2 \E\left[\Theta(1-\Theta)\right].
\end{align*}
According to Assumption \ref{eq:limsup_sigma}, there exist $A>0$ and $a>0$ such that for all $x>A$
\begin{align*}
2\frac{\sigma(x)^2}{x^2}+\frac{p(x)}{x^2}\int_{\R_+}z^2\pi(dz) -2\alpha x \E\left[\Theta(1-\Theta)\right]<-(a+g).
\end{align*}
Then, 
\begin{align*}
\mathcal{A}_s^{(t)}V(x) \leq -ax+ \mathbf{1}_{\lbrace x\leq A\rbrace}\left(2\frac{\sigma(x)^2}{x}+\frac{p(x)}{x}\int_{\R_+}z^2\pi(dz) -2\alpha x^2 \E\left[\Theta(1-\Theta)\right]\right),
\end{align*}
and according to Assumption \ref{ass_A}, there exists $d>0$ such that for every $x \geq 0$,
\begin{align*}
\mathcal{A}_s^{(t)}V(x) & \leq -ax+d.
\end{align*}

{\it ii)} Let $R>2da^{-1},$ where $a,d$ are given in {\it i)}. We will prove the minoration condition with $\nu = \delta_0$, where $\delta_0$ is the Dirac measure at $0$. Consider again $\bar{Y}$, defined as 
the unique strong solution to the SDE \eqref{eq:tildeY}. We recall that $\bar{Y_s}\geq Y_s^{(t)}$, for all $s\leq t$. Therefore for all $r<s\leq t$ and all Borel sets $A$ of $\mathbb{R}_+$,
\begin{align*}
\mathbb{P}\left(Y_s^{(t)}\in A\big|Y_r^{(t)}=x\right) \geq \mathbb{P}\left(Y_s^{(t)}=0\big|Y_r^{(t)} = x\right)\delta_0(A) \geq \mathbb{P}\left(\bar{Y}_s=0\big|\bar{Y}_r = x\right)\delta_0(A).
\end{align*} 
Next, notice that if $\bar{Y}^1, \bar{Y}^2$ are two solutions to \eqref{eq:tildeY} with respective initial conditions at time $r$ satisfying 
$\bar{Y}^1_r\leq \bar{Y}^2_r$, then $\bar{Y}^1_s\leq \bar{Y}^2_s$ for all $r\leq s\leq t$. Hence, for all $x\leq R$,
\begin{align*}
 \mathbb{P}\left(\bar{Y}_s=0\big|\bar{Y}_r = x\right)=\mathbb{P}\left(\bar{Y}_{s-r}=0\big|\bar{Y}_0 = x\right)\geq 
 \mathbb{P}\left(\bar{Y}_{s-r}=0\big|\bar{Y}_0 = R\right).
\end{align*}
Finally, using \cite[Theorem 3.3{\it iii)}]{companion} on $\bar{Y}$, there exists $\alpha_{s-r}>0$ such that 
$$
\mathbb{P}\left(\bar{Y}_{s-r}=0\big|\bar{Y}_0 = R\right)>\alpha_{s-r},
$$
which ends the proof.\end{proof}

\begin{proof}[Proof of Proposition \ref{prop:conv_auxi}] 
This result is a direct application of \cite[Proposition 3.3]{marguet2017law}, whose assumptions are satisfied thanks to Lemma \ref{lemma:Lyapunov_min}. 
\end{proof}
This convergence result allows us to establish a law of large numbers, linking asymptotically the behaviour of a 
typical individual, given by the auxiliary process $Y^{(t)}$, 
with the behaviour of the whole population.

\begin{thm}\label{thm:conv_auxi}
Suppose that Assumptions \ref{ass_E}, \ref{Ass_E'},  \ref{eq:limsup_sigma} hold, that $\int_{\R_+}z^6\pi(dz)<\infty$ and that
$\max(g,\beta)>q$.
Then for all bounded measurable functions $F:\mathbb{D}([0,T],\mathcal{X})\rightarrow\mathbb{R}$, for all $x_0,x_1\geq 0$,
\begin{align*}
\mathbf{1}_{\lbrace N_{t+T}\geq 1\rbrace}\left(\frac{\sum_{u\in V_{t+T}}F\left(X_{t+s}^{u},s\leq T\right)}{N_{t+T}}-
\mathbb{E}\left[F\left(Y_{t+s}^{(t+T)},s\leq T\right)\Big|Y_{0}^{(t+T)} = x_1 \right]\right)\xrightarrow[t\rightarrow +\infty]{\mathbb{L}_2(\delta_{x_0})} 0,
\end{align*}
where we recall that $N_{t} : = \mathrm{Card}(V_t)$.
\end{thm}

This result from \cite{marguet2017law} ensures that asymptotically, the trajectory of the traits of a sampling along its ancestral lineage corresponds to the trajectory of the auxiliary 
process. Hence, the study of the asymptotic behaviour of the proportion of individuals satisfying some properties, such as the proportion of infected individuals, 
is reduced to the study of the time-inhomogenous process $Y$.

The proof of Theorem \ref{thm:conv_auxi} is a direct application of \cite[Corollary 3.7]{marguet2017law}. Assumptions 2.1, 2.3 and 2.4 in \cite{marguet2017law} are 
satisfied thanks to Assumption \ref{ass_A}, using \eqref{mxtbis} and \eqref{mxst}, and the fact that $\beta>0$. 
We proved that Assumption 3.1 in \cite{marguet2017law} is verified in Lemma \ref{lemma:Lyapunov_min}. It remains to check that Assumptions 3.4 and 3.6 in \cite{marguet2017law} are satisfied. Note that in our case, the function $c(x)$ defined in \cite[Equation 3.3]{marguet2017law} 
is equal to $\max(g,\beta)-q$ and the first point of Assumption 3.4 in \cite{marguet2017law} is satisfied.

Next, we set some notations, introduced in \cite{marguet2017law}. For all $x,y\geq 0$ and $s\geq 0$, we define
\begin{equation*}
\varphi_{s}(x,y)=\sup_{t\geq s} \frac{m(x,0,s)m(y,s,t)}{m(x,0,t)},
\end{equation*}
(which does not depend on $q$)
and for all measurable functions $f:\mathbb{R}_+\rightarrow\mathbb{R}$ and $x\geq 0$,
\begin{align*}
Jf(x)=2\int_{0}^1f\left(\theta x\right)f\left((1-\theta)x\right)\kappa(d\theta).
\end{align*}
The next lemma amounts to check the second point of Assumption 3.4 in \cite{marguet2017law}. 
\begin{lemma}
Under Condition \ref{ass_E}, suppose that Assumptions \ref{Ass_E'}, \ref{eq:limsup_sigma} hold, and that $\int_{\R_+}z^6\pi(dz)<\infty$ . 
Then, for all $x\geq 0$,
\begin{equation*}
\sup_{t\geq 0}\E_{x}\left[r\left(Y_{t}^{(t)}\right)J\left((1\vee V(\cdot))\varphi_{t}\left(x,\cdot\right)\right)\left(Y_{t}^{(t)}\right)\right]<\infty.
\end{equation*}
\end{lemma}
\begin{proof}
Note that if $x=0$, $Y_t^{(t)}=0$ almost surely for all $t\geq 0$. Therefore, we only need to consider $x>0$. 
First assume that $\alpha>0$. Notice that for all $t\geq 0$ $x>0$, and $y\geq 0$, 
$$
\varphi_t(x,y)\leq \frac{\left(1+\frac{\alpha x}{|g-\beta|}\right)\left(1+\frac{\alpha y}{|g-\beta|}\right)}{\min\left(\frac{\alpha x}{|g-\beta|},1\right)},
$$
where we simplified by $e^{\max(g,\beta)t}$ in the fraction in the definition of $\varphi_s$.
Next, for all $x>0$, 
\begin{align*}
& \E_{x}\left[r\left(Y_{t}^{(t)}\right)J\left((1\vee V(\cdot))\varphi_{t}\left(x,\cdot\right)\right)\left(Y_{t}^{(t)}\right)\right]\\
& \leq \left(\frac{|g-\beta|+\alpha x}{\min\left(\alpha x,|g-\beta|\right)}\right)^2\E_{x}\left[\left(\alpha Y_{t}^{(t)}+\beta\right)
2\int_{0}^1\left(1\vee\theta Y_{t}^{(t)}\right)\left(1\vee(1-\theta) Y_{t}^{(t)}\right)\left(1+\frac{\alpha Y_{t}^{(t)}}{|g-\beta|}\right)^2\kappa(d\theta)\right]. 
\end{align*}
For all $k\geq 0$, we define
\begin{align*}
f_k^{(t)}(x,s)=\mathbb{E}_x\left[\left(Y_s^{(t)}\right)^k\right]
\end{align*}
and we end the proof of the lemma by showing that,
 $$\sup_{t\geq 0}\sup_{s\leq t}f_5^{(t)}(x,s)<\infty.$$
According to It\^o's formula, we have for $k\geq 2$,
\begin{align*}
f_k^{(t)}(x,s) = &  k\int_0^s \mathbb{E}_x\left[\left(Y_u^{(t)}\right)^{k-1}f_1\left(Y_u^{(t)},t-u\right)\right]du + k(k-1)\int_0^s \mathbb{E}_x\left[\sigma^2\left(Y_u^{(t)}\right)\left(Y_u^{(t)}\right)^{k-2}\right]du\\
& +\int_0^s \int_{\mathbb{R}_+}\mathbb{E}_x\left[f_2\left(Y_u^{(t)},t-u,\theta\right)\left(Y_u^{(t)}\right)^k(\theta^k-1)\right]\kappa(d\theta)du\\
& +\int_0^s \int_{\mathbb{R}_+}\mathbb{E}_x\left[f_3\left(Y_u^{(t)},t-u,z\right)\left(\left(Y_u^{(t)}+z\right)^k-\left(Y_u^{(t)}\right)^k-kz\left(Y_u^{(t)}\right)^{k-1}\right)\right]\pi(dz)du.
\end{align*}
Differentiating with respect to $s$ and using that for all $x\geq 0$ and $s\geq 0$, 
$$x^{k-1}f_1(x,s)\leq gx^k+2\sigma(x)^2x^{k-2}+p(x)x^{k-2}\int_{\R_+}z^2\pi(dz),\quad f_2(x,s,\theta)\geq 2\theta\alpha x\quad \text{for all} \quad \theta\in[0,1]$$
and 
$$f_3(x,s,z)\leq (x+z)p(x)/x \quad \text{for all} \quad z\geq 0,$$ 
and applying Taylor's formula with integral remainder, we obtain 
\begin{align*}
\partial_s f_k^{(t)}(x,s)\leq &  gk \mathbb{E}_x\left[\left(Y_s^{(t)}\right)^k\right]
 + k\mathbb{E}_x\left[\left((k+1)\sigma^2\left(Y_s^{(t)}\right)+p\left(Y_s^{(t)}\right)\int_{\R_+}z^2\pi(dz)\right)\left(Y_s^{(t)}\right)^{k-2}\right]\\
& - \int_{\mathbb{R}_+}\mathbb{E}_x\left[2\alpha \left(Y_s^{(t)}\right)^{k+1}\theta(1-\theta^k)\right]\kappa(d\theta)\\
& +k(k-1) \int_{\mathbb{R}_+}\int_0^z(z-u)\mathbb{E}_x\left[\frac{p\left(Y_s^{(t)}\right)}{Y_s^{(t)}}\left(Y_s^{(t)}+u\right)^{k-2}\left(Y_s^{(t)}+z\right)\right]du \pi(dz).
\end{align*}
Moreover, for all $y> 0$,
\begin{align*}
&\int_{\mathbb{R}_+}\int_0^z(z-u)\frac{p\left(y\right)}{y}\left(y+u\right)^{k-2}\left(y+z\right)du \pi(dz) \\
& \leq  \int_{\mathbb{R}_+}z^2\frac{p(y)}{y}\left(y+z\right)^{k-1} \pi(dz)
=  \frac{p(y)}{y}\int_{\mathbb{R}_+}z^2\sum_{l = 0}^{k-1}\dbinom{k-1}{l}y^lz^{k-1-l} \pi(dz)\\
& = \frac{p(y)}{y}\int_{\mathbb{R}_+}z^2\sum_{l = 0}^{k-2}\dbinom{k-1}{l}y^lz^{k-1-l} \pi(dz)+\frac{p(y)}{y^2}\left(\int_{\R_+}z^2\pi(dz)\right)y^{k}.
\end{align*}
Combining the last two inequalities, we get 
\begin{align*}
\partial_sf_k^{(t)}(x,s)\leq & k( A_t^{(k)}+B_t^{(k)}-C_t^{(k)}+D_t^{(k)}),
\end{align*}
with

$$ A_t^{(k)} = gf_k^{(t)}(x,s),\ 
B_t^{(k)} = \mathbb{E}_x\left[H(k,Y_s^{(t)})(Y_s^{(t)})^{k}\right],\ C_t^{(k)} =\frac{2\alpha}{k} \mathbb{E}\left[\Theta(1-\Theta^k)\right]f_{k+1}^{(t)}(x,s),
$$
$$D_t^{(k)} =(k-1) \int_{\mathbb{R}_+}\mathbb{E}_x\left[\frac{p(Y_s^{(t)})}{\left(Y_s^{(t)}\right)^2}\sum_{l = 0}^{k-2}\dbinom{k-1}{l}(Y_s^{(t)})^{l+1}z^{k+1-l}\right] \pi(dz),
$$
where
$$ H(k,y) = (k+1)\frac{\sigma^2(y)}{y^2}+k\frac{p(y)}{y^2}\int_{\R_+}z^2\pi(dz).$$
To end the proof we consider the case $k=5$.
According to Assumption \ref{eq:limsup_sigma} and using that $\sigma$ and $p$ are continuous (Assumption \ref{ass_A}), there exist $C_1,C_2>0$ and $A>0$ such that for all $y\geq 0$,
\begin{align*}
H(5,y)y^{5} = H(5,y)y^5\mathbf{1}_{\lbrace y>A\rbrace}+H(5,y)y^5\mathbf{1}_{\lbrace y\leq A\rbrace}&\leq C_1y^5\mathbf{1}_{\lbrace y>A\rbrace}
+C_2\mathbf{1}_{\lbrace y\leq A\rbrace}\leq C_1y^5+C_2.
\end{align*}
Moreover, $\limsup_{0^+}p(x)/x<\infty$ as $p(0)=0$ and $p$ is locally Lipschitz, $\limsup_\infty p(x)/x^2<\infty$ thanks to Assumption \ref{eq:limsup_sigma} 
and $\int_{\R_+}z^6\pi(dz)<\infty$, which yields
\begin{align*}
D_t ^{(5)}\leq C_3  (f_5^{(t)}(x,s) + 1),
\end{align*}
for some $C_3\geq 0$. Combining the last two inequalities, there exists $D_1>0$ such that
\begin{align*}
\partial_sf_5^{(t)}(x,s)\leq  D_1 (f_5^{(t)}(x,s)+1)-D_2f_{6}^{(t)}(x,s),
\end{align*}
where $D_2 = \frac{2\alpha}{5} \mathbb{E}\left[\Theta(1-\Theta^5)\right]$.
Applying Jensen inequality, we have $f_6^{(t)}(x,s)\geq f_5^{(t)}(x,s)^{6/5}$. Finally, we obtain
\begin{align*}
\partial_sf_5^{(t)}(x,s)\leq  F\left(f_5^{(t)}(x,s)\right),
\end{align*}
with $F(y) = D_1(y+1)-D_2y^{1+1/5}$. Any solution to the equation $y'=F(y)$ is bounded by $y(0)\vee x_0$, where $x_0 = \left(5D_1/6D_2\right)^5$
and so is $f_5^{(t)}(x,\cdot)$. It ends the proof for this case.
\end{proof}

Finally, we need to control the value of the second moment of the population size relatively to the square of its mean.
It corresponds to Assumption 3.6 in \cite{marguet2017law}.

\begin{lemma} \label{lem_esp_Nt2}
Suppose that Assumption \ref{ass_E} holds.
Then
for all $x> 0$, 
\begin{itemize}
\item[{\it (i)}] if $\beta> \max(g,q)$,
\begin{align*}
\mathbb{E}_{\delta_x}\left[N_t^2\right] \underset{t \to \infty}{\sim} C_1^2(x)e^{2(\beta-q)t},\quad 
\mathbb{E}_{\delta_x}^2\left[N_t\right] \underset{t \to \infty}{\sim} \left(1 + \frac{\alpha x}{\beta-g}\right)^2e^{2(\beta-q)t},
\end{align*}
where 
$$C_1^2(x)=1 + \frac{\alpha x}{2\beta-g-q} - \frac{\alpha^2 x^2}{(\beta-g)^2}
+\left(1+\frac{\alpha x}{\beta-g}\right)\left(\frac{2\alpha x}{\beta-g}+\frac{\beta+q}{\beta-q}\right)+\frac{\alpha x}{g-\beta}\frac{(\beta+q)}{2\beta-g-q}.$$
\item[{\it (ii)}] if $g>\max(\beta,q)$ and $\alpha>0$,
\begin{align*}
\mathbb{E}_{\delta_x}\left[N_t^2\right] \underset{t \to \infty}{\sim}\mathbb{E}_{\delta_x}^2\left[N_t\right] \underset{t \to \infty}{\sim} 
\left(\frac{\alpha x}{\beta-g}\right)^2e^{2(g-q)t},
\end{align*}
\end{itemize}
\end{lemma}

\begin{proof}
According to It\^o's formula, we have for all $t\geq 0$ and $x> 0$,

\begin{align*}
\frac{d\mathbb{E}_{\delta_x}\left[N_t^2\right]}{dt}& =  \mathbb{E}_{\delta_x}\left[\left(\alpha x e^{(g-q)}t+\beta N_t\right)\left(1+2N_t\right)\right]
+  \mathbb{E}_{\delta_x}\left[q N_t\left(1-2N_t\right)\right]\\
 & =\alpha xe^{(g-q)}t+2\alpha x e^{(g-q)}t\mathbb{E}_{\delta_x}[N_t]+(\beta+q)\mathbb{E}_{\delta_x}[N_t]+2(\beta-q) \mathbb{E}_{\delta_x}\left[N_t^2\right]. 
\end{align*}
Using \eqref{mxst} and variation of constants we obtain
\begin{align*}
\E_{\delta_x}\left[N_t^2\right] = & e^{2(\beta-q)t} +  \alpha x\frac{\left(e^{(g-q)t}-e^{2(\beta-q)t}\right)}{g-2\beta+q} + \frac{2\alpha^2 x^2}{g-\beta}\frac{\left(e^{2(g-q)t}-e^{2(\beta-q)t}\right)}{2(g-\beta)}\\
&+2\alpha x\left(1-\frac{\alpha x}{g-\beta}\right)\frac{\left(e^{(g+\beta-2q)t}-e^{2(\beta-q)t}\right)}{g-\beta}+(\beta+q)\frac{\alpha x}{g-\beta}\frac{\left(e^{(g-q)t}-e^{2(\beta-q)t}\right)}{g-2\beta+q}\\
& - \left(1-\frac{\alpha x}{g-\beta}\right)(\beta+q)\frac{\left(e^{(\beta-q)t}-e^{2(\beta-q)t}\right)}{\beta-q}.
\end{align*}
Therefore, 
$$
\begin{array}{ll}
\text{if }g>\max(\beta,q),& \E_{\delta_x}\left[N_t^2\right]\underset{t\rightarrow +\infty}{\sim}\frac{\alpha^2x^2}{(g-\beta)^2}e^{2(g-q)t}\\
\text{if }\beta>\max(g,q), & \E_{\delta_x}\left[N_t^2\right]\underset{t\rightarrow +\infty}{\sim}C_1^2(x)e^{2(\beta-q)t}.
\end{array}
$$
Moreover, 
\begin{align*}
\E_{\delta_x}\left[N_t\right]^2 = \frac{\alpha^2 x^2}{(g-\beta)^2}e^{2(g-q)t}+ \left(1-\frac{\alpha x}{g-\beta}\right)^2e^{2(\beta-q)t}-\frac{\alpha x }{\beta-g}\left(1-\frac{\alpha x}{g-\beta}\right)e^{g+\beta-2q},
\end{align*}
so that
$$
\begin{array}{ll}
\text{if }g>\max(\beta,q),& \E_{\delta_x}\left[N_t\right]^2\underset{t\rightarrow +\infty}{\sim}\frac{\alpha^2x^2}{(g-\beta)^2}e^{2(g-q)t}\\
\text{if }\beta>\max(g,q), & \E_{\delta_x}\left[N_t\right]^2\underset{t\rightarrow +\infty}{\sim}\left( 1+\frac{\alpha x }{\beta-g}\right)^2e^{2(\beta-q)t}.
\end{array}
$$
\end{proof}

We checked that all required assumptions to apply \cite[Corollary 3.4]{marguet2017law} are satisfied. This ends the proof of Theorem \ref{thm:conv_auxi}.

\subsubsection{Proof of Proposition \ref{prop:temps_long_auxi}}
 
We first prove point {\it ii)}.
The first step consists in proving that for every $x\geq 0$,
\begin{equation} \label{abs_finitetime_Yt}
 \P_x(\exists t<\infty, Y^{(t)}_t=0)=1. 
\end{equation}
A direct application of (6.3) in \cite[Theorem 6.2]{companion} is not possible 
because of the time-inhomogeneity of the process $Y^{(t)}$. Therefore, we couple $Y^{(t)}$
with a process $(\hat{Y}_s,s\geq 0)$ defined as the unique strong solution to
\begin{align*}
\hat{Y}_s = & Y_0^{(t)}+ \int_0^s\hat{g}(\hat{Y}_u)du + \int_0^s\sqrt{2\sigma^2(\hat{Y}_u)}dB_u 
+ \int_0^s\int_0^1(\theta-1)\int_{0}^{\hat{r}(\hat{Y}_u,\theta)}\hat{Y}_u N(du,d\theta,dx)\\
& + \int_0^s \int_0^\infty\int_0^{f_3(\hat{Y}_u,t-s,z)}z\widetilde{Q}(du,dz,dx),
\end{align*}
where $f_3(y,t-s,z)= 0$ for $s \geq t$, $B$, $N$ and $\tilde{Q}$ are the same as in \eqref{eq:EDSauxi} and for $x,s\geq 0$, $0 \leq \theta\leq 1$,
\begin{align*}
&f_1(x,s)\leq \hat{g}(x): =
gx+ \left(\mathbf{1}_{\lbrace\beta>g\rbrace}\frac{\alpha}{\beta-g+ \alpha x}+\mathbf{1}_{\lbrace g>\beta\rbrace}\frac{1}{x}\right)\left( 2\sigma^2(x)+ p(x)\int_{\R_+}z^2\pi(dz)\right)\\ 
&f_2(x,s,\theta)\geq \hat{r}(x,\theta):= 2\theta(\alpha x+\beta).
\end{align*}
Then, for all $t\geq 0$ and $0\leq s \leq t$, 
$
Y_s^{(t)} \leq \hat{Y}_s. 
$
In particular, for all $t\geq 0$,
\begin{align}\label{couplage_t}
Y_t^{(t)} \leq \hat{Y}_t. 
\end{align}
According to Lemma \ref{chgt-eds}, there exists a Poisson point measure $N'$ on $\mathbb{R}_+\times [0,1]\times\mathbb{R}_+$ with intensity 
$du\otimes\hat{\kappa}(d\theta)\otimes dx$ where $\hat{\kappa}(d\theta)=2\theta\kappa(d\theta)$, such that $\hat{Y}$ is also a strong pathwise solution to
\begin{multline*}
\hat{Y}_s =  Y_0^{(t)}+ \int_0^s\hat{g}\left(\hat{Y}_u\right)du + \int_0^s\sqrt{2\sigma^2\left(\hat{Y}_u\right)}dB_u 
+ \int_0^s\int_{0}^{r(\hat{Y}_u)}\int_0^1(\theta-1)\hat{Y}_u N'(du,dx,d\theta)\\
 + \int_0^s \int_0^\infty\int_0^{f_3(\hat{Y}_u,t-s,z)}z\widetilde{Q}(du,dz,dx),
\end{multline*}
where we used that $\int_0^12\theta\kappa(d\theta)=1$ because $\kappa$ is symmetrical with respect to $1/2$.

Let us check that despite the fact that the jump rate $f_3$ depends on jump size and time, 
(6.3) in \cite[Theorem 6.2]{companion} holds for $\hat{Y}$. First, we need to prove that \cite[Lemma 7.2]{companion} still holds under these modifications.
Let us choose $x_0>0$ and introduce $\tau^-(x_0):= \inf \{t \geq 0, \hat{Y}_t \leq x_0\}$.
Following the same steps as in the proof of \cite[Lemma 7.2]{companion} we have,
\begin{align}\label{eq:lnhatY}\nonumber
\ln\hat{Y}_{t\wedge \tau^-(x_0)} = &  \ln\hat{Y}_{0}+\int_{0}^{t\wedge \tau^-(x_0)}\frac{\hat{g}(\hat{Y}_s)}{\hat{Y}_s}ds-\int_{0}^{t\wedge \tau^-(x_0)}
\frac{\sigma^2(\hat{Y}_s)}{\hat{Y}_s^2}ds
+\E[\ln \hat{\Theta}]\int_{0}^{t\wedge \tau^-(x_0)}r(\hat{Y}_s)ds \\
+&\int_{0}^{t\wedge \tau^-(x_0)} \int_0^\infty f_3(\hat{Y}_s,t-s,z)\left[\ln\Big(1+\frac{z}{\hat{Y}_{s}}\Big)-\frac{z}{\hat{Y}_s}\right]\pi(dz)ds+M_{t\wedge \tau^-(x_0)}
\end{align}
where $\hat{\Theta}$ is a random variable with law $\hat{\kappa}$ and $(M_{s \wedge \tau^-(x_0)}, s \geq 0)$ is a martingale.
Let us check that the condition {\bf (LSG)} of \cite{companion} is satisfied, {\it i.e.} that 
\begin{enumerate}[label=\bf{(LSG)}]
\item \label{LSG} There exist $\eta>0$ and $x_0>0$ such that for all $x>x_0$,
\begin{equation*} 
\hat{H}(x):=\frac{\hat{g}(x)}{x} - \frac{\sigma^2(x)}{x^2}+r(x)\E\left[\ln \hat{\Theta} \right] 
+p(x)\int_{\mathbb{R}_+} \left(\ln\Big(1+\frac{z}{x}\Big)-\frac{z}{x}\right)\pi(dz)\leq -\eta.
\end{equation*}
\end{enumerate}
We have
\begin{align*}
\hat{H}(x) &\leq g + \left(\frac{\mathbf{1}_{\{\beta>g\}}\alpha x}{\beta-g + \alpha x} + \mathbf{1}_{\{g \geq\beta\}}\right)\left(2\frac{\sigma^2(x)}{x^2} +
\frac{p(x)}{x^2}\int_{\R_+}z^2\pi(dz)
\right)+2(\alpha x+\beta)\E\left[\Theta\ln \Theta \right]\\
& \leq g + C\left(\frac{\mathbf{1}_{\{\beta>g\}}\alpha x}{\beta-g + \alpha x} + \mathbf{1}_{\{g \geq\beta\}}\right)+2(\alpha x+\beta)\E\left[\Theta\ln \Theta \right],
\end{align*}
where $C$ is a finite constant, according to Assumptions \ref{ass_A} and \ref{eq:limsup_sigma}.
As $\E\left[\Theta\ln \Theta \right]<0$, we deduce that the condition \ref{LSG} is satisfied.

Moreover, notice that the dependence on jump size and time for the jump rate $f_3$ does not modify 
the proof of \cite[Lemma 7.2]{companion}, as the last term in \eqref{eq:lnhatY} is still negative.
Finally, we check that \cite[Theorem 4.1{\it i)}]{companion} holds for $\hat{Y}$. First, similarly as in the proof of Proposition \ref{prop_extin_auxi}, we check that \cite[Theorem 4.1{\it i)}]{companion} 
still holds if the rate of positive jumps depends on time and jumps sizes and if for $x,z\geq 0$ and $s\leq t$, it satisfies \eqref{bound_f3}.
Note that this condition is satisfied by the positive jump rate $f_3$ of $\hat{Y}$. Adapting the proof of Proposition \ref{prop_extin_auxi} to this case, we get that 
\eqref{eq:limsup_p} becomes
\begin{align*}
\limsup_{x\rightarrow\infty}x^{-2}\int_{\mathbb{R}_+}p_t(x,s,z) z^2\left(\int_0^1(1+zx^{-1}v)^{-1-a}(1-v)dv\right)\pi(dz)\\
\leq C\limsup_{x\rightarrow\infty}\frac{p(x)}{x^2}\int_{\mathbb{R}_+}(z^2+z^3) \pi(dz)<\infty
\end{align*}
for some finite and positive $C$, combining Assumption \ref{eq:limsup_sigma}, \ref{ass_E} and the fact that $\int_{\mathbb{R}_+}z^6\pi(dz)<\infty$. 
Concluding as in the beginning of the proof of 
Proposition \ref{prop_extin_auxi}, we get the desired generalization of \cite[Theorem 4.1{\it i)}]{companion}. 
To apply this generalized result to $\hat{Y}$ we need to check that condition \ref{SNinfinity} is satisfied for $\hat{Y}$. And it is the case according to Assumption \ref{eq:limsup_sigma} as
\begin{align*}
\frac{\hat{g}(x)}{x}
= g  + \left(\mathbf{1}_{\{\beta>g\}}\frac{\alpha x}{\beta-g + \alpha x} + \mathbf{1}_{\{g>\beta\}}\right)\left(2\frac{\sigma^2(x)}{x^2} +
\frac{p(x)}{x^2}\int_{\R_+}z^2\pi(dz)
\right).
\end{align*}

Therefore, \cite[Lemma 7.2]{companion} holds for $\hat{Y}$.

The proof of \cite[Eq.~(6.3)]{companion} (p.23 of \cite{companion}) requires \cite[Eq.(7.18)]{companion}. To prove \cite[Eq.(7.18)]{companion} for $\hat{Y}$, the only difference is that we have to deal with the dependence on the jump size of the jump rate $f_3$ to obtain a lower bound on the probability to have no positive jump during a time interval of the form $[0,t\wedge\tau^+(y)\wedge \tau^-(x)]$ with $0< x < y$.
Hence the idea is to bound the expectation of the sum of positive jumps on $[0,t\wedge\tau^+(y)\wedge \tau^-(x)]$ and use Markov inequality. Let $T= \tau^+(y)\wedge \tau^-(x)$. Then, for any $y_0\in (x,y)$,
\begin{align*}
\P_{y_0}\left(\hat{Y}_{t\wedge T}\geq y\right)\leq \P_{y_0}&\left(\int_0^{t\wedge T} g(\hat{Y}_s)ds +I_Q(t\wedge T)+\int_0^{t\wedge T} \sqrt{2\sigma^2(\hat{Y}_s)}dB_s\geq y-y_0\right),
\end{align*} 
where
\begin{align}\label{eq:twojumpint}\nonumber
I_Q(t\wedge T)&:=\int_0^{t\wedge T}\int_{\R_+}\int_0^{f_3(\hat{Y}_s,t-s,z)}zQ(ds,dz,dx)\\\nonumber
&\leq \int_0^{t\wedge T}\int_{0}^{\hat{Y}_s}\int_0^{2p(\hat{Y}_s)}zQ(ds,dz,dx) + \int_0^{t\wedge T}\int_{\hat{Y}_s}^\infty\int_0^{2zp(\hat{Y}_s)/\hat{Y}_s}zQ(ds,dz,dx)\\
&\leq \int_0^{t\wedge T}\int_{\R_+}\int_0^{2p(\hat{Y}_s)}zQ(ds,dz,dx) + \int_0^{t\wedge T}\int_{\R_+}\int_0^{2zp(\hat{Y}_s)/\hat{Y}_s}zQ(ds,dz,dx).
\end{align}
Then, as in \cite{companion} p.17, if we denote by $J(t,x,y)$ the event of having no
positive jumps due to the first integral in \eqref{eq:twojumpint}, we have for all $y_0\in (x,y)$,
\begin{align*}
\P_{y_0}(J(t,x,y))\geq e^{-2p(y)},
\end{align*}
because $p$ is non-decreasing according to Assumption \ref{ass_A}. Next,
\begin{align*}
& \P_{y_0}\left(\hat{Y}_{t\wedge T}\geq y, J(t,x,y)\right)\\
& \leq \P_{y_0}\left(\int_0^{t\wedge T} g(\hat{Y}_s)ds +\int_0^{t\wedge T}\int_0^{2p(\hat{Y}_s)/\hat{Y}_s}\int_{\R_+}z\overline{Q}(ds,dx,dz)+\int_0^{t\wedge T} \sqrt{2\sigma^2(\hat{Y}_s)}dB_s\geq y-y_0\right),
\end{align*}
where $\overline{Q}$ is a Poisson Point measure with intensity $ds\otimes dx\otimes z\pi(dz)$ (see Appendix \ref{app:generalization SDE}). Considering as before the event that there is no positive jumps associated to $\overline{Q}$, we get
\begin{align*}
\P_{y_0}\left(\hat{Y}_{t\wedge T}\geq y\right)\leq 2-e^{-2tp(y)}-e^{-2tp(y)/y}+ \P_{y_0}\left(\int_0^{t\wedge T} g(\hat{Y}_s)ds +\int_0^{t\wedge T} \sqrt{2\sigma^2(\hat{Y}_s)}dB_s\geq y-y_0\right).
\end{align*}
We conclude as in \cite{companion}, using that $\sup_{x \leq v \leq z}\hat{g}(v)<\infty$ according to Assumption \ref{ass_A} and $\inf_{x \leq v \leq y}r(v)>0$ holds under Assumption \ref{ass_E}.

Hence \cite[Eq.~(6.3)]{companion} holds for $\hat{Y}$, and \eqref{couplage_t} gives \eqref{abs_finitetime_Yt}.
Applying Theorem~\ref{thm:conv_auxi} to the function $F(X_{t+s}^u, s \leq T) = \mathbf{1}_{\{X_{t+T}^u>0\}}$ concludes the proof of point {\it ii)}.\\

Let us now prove the first part of point {\it i)}. Let $K>0$. First, as the assumptions of Proposition \ref{prop:conv_auxi} are satisfied, $Y_t^{(t)}$ converges to a nondegenerate random variable, 
which implies that for all $x\geq 0$
\begin{align}
\label{conv_K}
\lim_{K\rightarrow \infty}\lim_{t\rightarrow \infty}\mathbb{P}(Y_t^{(t)}>K|Y_0^{(t)}=x)=0.
\end{align}
Let $\varepsilon>0$. By Markov inequality,
\begin{align}\label{eq:markov_mto}
\mathbb{P}_{\delta_x}\left(  \frac{\sum_{u \in V_t} \mathbf{1}_{\{X_t^u >K\}} }{m(x,0,t)}>\eps \right)
&\leq \eps^{-1}\E_{\delta_x}\left[\frac{\sum_{u \in V_t} \mathbf{1}_{\{X_t^u >K\}}}{m(x,0,t)}\right]
= \eps^{-1}\mathbb{P}_x\left(Y_t^{(t)}>K\right)
\end{align}
where the last inequality comes from \eqref{eq:mto} applied to the function 
$$F((X_s^u,s \leq t))= \mathbf{1}_{\{X_t^u >K\}}.$$
Thus, taking the limit in \eqref{eq:markov_mto} in $t$ and $K$ yields the result.

We end with the second part of point {\it i)}. Conditions of Theorem \ref{thm:conv_auxi} are satisfied. Hence, taking again the same function $F$ we obtain 
for any $x_0,x_1,K \geq 0$, 
\begin{equation*}
 \lim_{t \to \infty} \E_{\delta_{x_0}} \left[ \mathbf{1}_{\{N_t\geq 1\}} \left|
 \frac{\sum_{u \in V_t} \mathbf{1}_{\{X_t^u >K\}} }{N_t}-\mathbb{P}_{x_1}(Y_t^{(t)}>K) \right|^2 \right] = 0
\end{equation*}
Let $\eps>0$. From \eqref{conv_K} we know that there exists $t_0,K_0 \geq 0$ such that for any $K \geq K_0$ and $t \geq t_0$,
$$ \mathbb{P}_{x_1}(Y_t^{(t)}>K) \leq \eps/2. $$
We thus obtain the following series of inequalities:
\begin{align*}
& \P_{x_0} \left(\mathbf{1}_{\{N_t\geq 1\}}
 \frac{\sum_{u \in V_t} \mathbf{1}_{\{X_t^u >K\}} }{N_t}> \eps\right) \\
&= \P_{x_0} \left(\mathbf{1}_{\{N_t\geq 1\}} \left(
 \frac{\sum_{u \in V_t} \mathbf{1}_{\{X_t^u >K\}} }{N_t}-\mathbb{P}_{x_1}(Y_t^{(t)}>K)\right) > \eps- \mathbf{1}_{\{N_t\geq 1\}} \mathbb{P}_{x_1}(Y_t^{(t)}>K)\right)\\ 
& \leq \P_{{x_0}} \left(\mathbf{1}_{\{N_t\geq 1\}} \left(
 \frac{\sum_{u \in V_t} \mathbf{1}_{\{X_t^u >K\}} }{N_t}-\mathbb{P}_{x_1}(Y_t^{(t)}>K)\right) > \eps/2\right)\\
& \leq\P_{{x_0}} \left(\mathbf{1}_{\{N_t\geq 1\}} \left(
 \frac{\sum_{u \in V_t} \mathbf{1}_{\{X_t^u >K\}} }{N_t}-\mathbb{P}_{x_1}(Y_t^{(t)}>K)\right)^2 > \eps^2/4\right)\\
 &\leq \frac{4}{\eps^2}  \E_{x_0} \left[ \mathbf{1}_{\{N_t\geq 1\}} \left|
 \frac{\sum_{u \in V_t} \mathbf{1}_{\{X_t^u >K\}} }{N_t}-\mathbb{P}_{x_1}(Y_t^{(t)}>K) \right|^2 \right] \underset{t \to \infty}{\to} 0.
\end{align*}
As the first term is increasing with $K$ we obtain the desired result.

\subsection{Proof of Proposition \ref{prop_constant_diff}}

Proposition \ref{prop_constant_diff} is a consequence of the following two lemmas.

\begin{lemma}\label{maj_mto}
 Assume that there exists a real number $\gamma$ such that $r(x)-q(x)\leq \gamma$ for any $x \in \R_+$, and let $f$ be a nonnegative measurable function on $\R_+$.
Then for $x,t \geq 0$,
$$ \E_{\delta_{x}}\left[\sum_{u\in V_{t}}f\left(X_{t}^{u}\right)\right]\leq e^{\gamma t}\E_{x}\left[f\left(Y_{t}\right)\right],
 $$
where we recall that $Y$ is the unique strong solution to the SDE \eqref{SDE_Y_diff_ct}.
\end{lemma}

\begin{lemma}\label{lem_gene}
Let $\zeta>0$.
\begin{itemize}
  \item[i)]  Assume that Assumption \ref{ass_expl_tps_fini} holds for some $a>1$ such that $\E[\Theta^{1-a}]<\infty$.
Then for any $x>0$,  
  $$ \lim_{t \to \infty} \E_{\delta_x}\left[\sum_{u \in V_t} \left(X_t^u\vee \zeta\right)^{1-a} \right] =0. $$
  \item[ii)]   Assume that Assumption \ref{ass_abs_tps_fini} holds for some $a<1$.
Then for any $x>0$,  
  $$ \lim_{t \to \infty} \E_{\delta_x}\left[\sum_{u \in V_t} \left(X_t^u\wedge \zeta\right)^{1-a} \right] =0. $$
 \end{itemize}
\end{lemma}

\begin{proof}[Proof of Lemma \ref{maj_mto}]
We will use a normalisation of the population process similar to the one leading to the auxiliary process and relying on this assumption. Let
$$
R_{0,t}f(x) = \E\left[\sum_{u\in V_t} f(X_t^u)|Z_0 = \delta_x\right]
$$
be the first moment semigroup of $Z$, for $x,t\geq 0$. 
Then we have
\begin{align*}
e^{-\gamma t}& R_{0,t}f(x) = f(x) + \int_0^t \int_{\R_+}\left(\mathcal{G}f(\mathfrak{x})e^{-\gamma r}-\gamma f(\mathfrak{x})e^{-\gamma r}\right)R_{0,r}(x,d\mathfrak{x})dr\\
& +\int_0^t\int_{\R_+}e^{-\gamma r}\left(r(\mathfrak{x})\int_0^1(f(\theta \mathfrak{x})+f((1-\theta)\mathfrak{x})-f(\mathfrak{x}))\kappa(d\theta)-
q(\mathfrak{x})f(\mathfrak{x})\right)R_{0,r}(x,d\mathfrak{x})dr\\
= & f(x) + \int_0^t \int_{\R_+}\left(\mathcal{G}f(\mathfrak{x})e^{-\gamma r}-\gamma f(\mathfrak{x})e^{-\gamma r}\right)R_{0,r}(x_0,d\mathfrak{x})dr\\
& +\int_0^t\int_{\R_+}e^{-\gamma r}\left(2r(\mathfrak{x})\int_0^1(f(\theta \mathfrak{x})-f(\mathfrak{x}))\kappa(d\theta)+(r(\mathfrak{x})-q(\mathfrak{x}))f(\mathfrak{x})
\right)R_{0,r}(x,d\mathfrak{x})dr,
\end{align*}
and using that $r(x)-q(x)\leq \gamma$ for all $x\geq 0$, we obtain
\begin{align*}
\tilde{R}_{0,t}f(x) \leq & f(x) + \int_0^t \int_{\R_+}\mathcal{G}f(\mathfrak{x})\tilde{R}_{0,r}(x,d\mathfrak{x})dr +\int_0^s\int_{\R_+}2r(\mathfrak{x})\int_0^1(f(\theta \mathfrak{x})-f(\mathfrak{x}))\kappa(d\theta)
\tilde{R}_{0,r}(x,d\mathfrak{x})dr,
\end{align*}
where $\tilde{R}_{0,t}f(x) = e^{-\gamma t} R_{0,t}f(x).$
Finally,
\begin{align}\label{eq:ineq_MTO}
\tilde{R}_{0,t}f(x) = \E\left[\sum_{u\in V_t} f(X_t^u)|Z_0 = \delta_x\right]e^{-\gamma t}\leq \E_{x}\left[f\left(Y_t\right)\right],
\end{align}
where $Y$ is the unique strong solution to the SDE
\eqref{SDE_Y_diff_ct}.
\end{proof}

\begin{proof}[Proof of Lemma \ref{lem_gene}]
To begin with, let us prove using a coupling argument that under the assumptions of point~$i)$, for all $y>0$, $\P_y\left(\tau^-(0)<\infty\right)=0$, 
where $\tau^-(0)=\inf\left\{t\geq 0, Y_t=0\right\}$. Let $K> 0$. We consider the process $\tilde{Y}$ defined as the unique strong solution to
\begin{align*}
\tilde{Y}_t= x& + \int_0^tg(\tilde{Y}_s)ds+ \int_0^t \sqrt{2 \sigma^2(\tilde{Y}_s) }dB_s+ \int_0^t \int_0^{p(\tilde{Y}_{s^-})}
\int_{\mathbb{R}_+}z\widetilde{Q}(ds,dx,dz)\\
&+\int_0^t\int_0^{\tilde{Y}_{s^-}}\int_{\mathbb{R}_+}zR(ds,dx,dz)+\int_0^t\int_0^{2\overline{r}_K} \int_0^1  (\theta-1)\tilde{Y}_{s^-}N(ds,dx,d\theta),
\end{align*}
where $B,\widetilde{Q}$ and $N$ are the same as in \eqref{SDE_Y_diff_ct} and $\overline{r}_K=\sup_{x\in[0,K]}r(x)$. Then, as $p$ is a non-decreasing function, 
\begin{align}\label{eq:coupling}
\tilde{Y}_t\leq Y_t\quad \text{ for all }\quad t\leq \tau^+(K):=\inf\left\{t\geq 0, Y_t\geq K\right\}.
\end{align}

Let $a>1$ be as in Assumption \ref{ass_expl_tps_fini}. We thus have
$$ \tilde{G}_a(u)= G_a(u)- 2(\bar{r}_K-r(u))\E[\Theta^{1-a}-1] \geq \gamma' + o(\ln u), \quad (u \to 0). $$
Moreover, we can check that in the presence of stable positive jumps, the proof of \cite[Theorem 3.3{\it i)}]{companion} is not modified.
We thus obtain that for all $y>0$,
$
\P_y\left(\tilde{\tau}^-(0)<\infty\right)=0,
$
where $\tilde{\tau}^-(0)=\inf\left\{t\geq 0, \tilde{Y}_t=0\right\}$. Then, from \eqref{eq:coupling} we get $\P_y\left(\tau^-(0)<\tau^+(K)\right)=0$ and letting 
$K$ tend to infinity yields 
\begin{equation}\label{non_ext_gene}
\P_y\left(\tau^-(0)<\infty\right)=0.
\end{equation}
Now,  let $\zeta>0$ and $\eps>0$.
Then we have from \eqref{eq:ineq_MTO}, for every $t\geq 0$, $x>0$,
 \begin{align*}
  \E_{\delta_x} \left[ \sum_{u \in V_t} \left(X_t^u \vee \zeta \right)^{1-a} \right] & \leq e^{\gamma t}  \E_x \left[ \left(Y_t \vee \zeta \right)^{1-a} \right]\\
   & \leq e^{\gamma t}  \E_x \left[\mathbf{1}_{\{t \leq \tau^-(\eps) \wedge \tau^+(1/\eps)\}}Y_t^{1-a} \right]
   + e^{\gamma t}  \E_x \left[\mathbf{1}_{\{t > \tau^-(\eps) \wedge \tau^+(1/\eps)\}}\left(Y_t \vee \zeta \right)^{1-a} \right].
 \end{align*}
The first term can be bounded as follows:
 \begin{align*}
  e^{\gamma t}  \E_x \left[\mathbf{1}_{\{t \leq \tau^-(\eps) \wedge \tau^+(1/\eps)\}}Y_t^{1-a} \right] 
  & = e^{(\gamma-\gamma') t}  \E_x \left[Y_{t \wedge \tau^-(\eps) \wedge \tau^+(1/\eps)}^{1-a} e^{\gamma'(t \wedge \tau^-(\eps) \wedge \tau^+(1/\eps))}
  \mathbf{1}_{\{t \leq \tau^-(\eps) \wedge \tau^+(1/\eps)\}}\right]\\
  & \leq e^{(\gamma-\gamma') t}  \E_x \left[ Y_{t \wedge \tau^-(\eps) \wedge \tau^+(1/\eps)}^{1-a} e^{\int_0^{t \wedge \tau^-(\eps) \wedge \tau^+(1/\eps)}
  G_{a}(Y_s)ds} \right]\\
  & = x e^{(\gamma-\gamma') t},
 \end{align*}
using that $G_a(x)\geq \gamma'$ for all $x\geq 0$ and the martingale property. The second term may be divided into two parts as follows:
 \begin{align*}
  \E_x \left[\mathbf{1}_{\{t > \tau^-(\eps) \wedge \tau^+(1/\eps)\}}\left(Y_t \vee \zeta \right)^{1-a} \right] 
  & =  \E_x \left[\mathbf{1}_{\{\tau^-(\eps) <t\}}\left(Y_t \vee \zeta \right)^{1-a} \right] 
  +  \E_x \left[\mathbf{1}_{\{ \tau^+(1/\eps) < t \leq \tau^-(\eps) \}}\left(Y_t \vee \zeta \right)^{1-a} \right] .
\\ & \leq  \E_x \left[\mathbf{1}_{\{\tau^-(\eps) <t\}}\left(Y_t \vee \zeta \right)^{1-a} \right] 
  +  \E_x \left[\mathbf{1}_{\{ \tau^+(1/\eps) < t  \}}\left(Y_t \vee \zeta \right)^{1-a} \right] .
 \end{align*}
For any fixed $t\geq 0$, as $a>1$,
$$ \mathbf{1}_{\{\tau^-(\eps) <t\}}\left(Y_t \vee \zeta \right)^{1-a}\leq \zeta ^{1-a} $$
where $\zeta^{1-a}$ is finite and does not depend on $\eps$, and we know thanks to \eqref{non_ext_gene} that 
$$  \lim_{\eps \to 0}\mathbf{1}_{\{\tau^-(\eps) <t\}} = 0 \quad \text{almost surely.} $$
Hence by the dominated convergence theorem, we obtain that
$$\lim_{\eps \to 0}\E_x \left[\mathbf{1}_{\{\tau^-(\eps) <t\}}\left(Y_t \vee \zeta \right)^{1-a} \right]  = 0 \quad \text{almost surely.} $$
Let us finally consider the last term. 
First, notice that for every $\eps>0$
\begin{equation}\label{maj_CVD}
\mathbf{1}_{\{ \tau^+(1/\eps) < t \}}\left(Y_t \vee \zeta \right)^{1-a}\leq \zeta^{1-a}
\end{equation}
where $\zeta^{1-a}$ is finite and does not depend on $\eps$.
Now, let us consider the sequence of stopping times $(\tau^+(1/\eps), \eps>0)$. This sequence increases when $\eps$ decreases, and there exists 
$\tau^+(\infty)$, which may be infinite, defined by
\begin{equation} \label{deftauinfty} \lim_{\eps \to 0} \tau^+(1/\eps)=: \tau^+(\infty). \end{equation}
There are two cases:
\begin{itemize}
 \item Either $\tau^+(\infty)\leq t$. In this case, $Y_t=\infty$ and 
 $$  \mathbf{1}_{\{ \tau^+(1/\eps) < t  \}}\left(Y_t \vee \zeta \right)^{1-a} = 0.$$
 \item Or $\tau^+(\infty)> t$. In this case, there exists $\eps_0>0$ such that for any $\eps \leq \eps_0$, $\tau^+(1/\eps) \geq  t$, and for 
 such an $\eps$, 
  $$  \mathbf{1}_{\{ \tau^+(1/\eps) < t  \}}\left(Y_t \vee \zeta \right)^{1-a} = 0.$$
\end{itemize}
We deduce that for any fixed $t\geq 0$
$$ \lim_{\eps \to 0} \mathbf{1}_{\{ \tau^+(1/\eps) < t  \}}\left(Y_t \vee \zeta \right)^{1-a} = 0 \quad \text{almost surely.} $$
From \eqref{maj_CVD} we may apply the dominated convergence theorem and obtain
$$ \lim_{\eps \to 0} \E_x \left[\mathbf{1}_{\{ \tau^+(1/\eps) < t  \}}\left(Y_t \vee \zeta \right)^{1-a} \right] = 0. $$

To sum up, we proved that for any $t \geq 0$, $x>0$, and $\eps>0$, 
 \begin{align*}
  \E_{\delta_x} \left[ \sum_{u \in V_t} \left(X_t^u \vee \zeta \right)^{1-a} \right] 
   & \leq x e^{(\gamma-\gamma') t} + e^{\gamma t}  \E_x \left[\mathbf{1}_{\{t > \tau^-(\eps) \wedge \tau^+(1/\eps)\}}\left(Y_t \vee \zeta \right)^{1-a} \right] 
   \xrightarrow[\eps \to 0]{} x e^{(\gamma-\gamma') t} .
 \end{align*}
Letting $t$ tend to infinity ends the proof of point {\it i)}, as $\gamma<\gamma'$.
 
Let us now turn to the proof of point $ii)$. It is similar in spirit to the proof of point $i)$. 
Let $a<1$ and $\gamma'>0$ be such that Assumption \ref{ass_abs_tps_fini} holds. 
First of all, as $G_a(x)\geq \gamma'$ for all $x\geq 0$, \ref{SNinfinity} holds for $Y$ (defined as before as the unique strong solution to \eqref{SDE_Y_diff_ct}) 
and thus according to \cite[Theorem 4.1{\it i)}]{companion}, for all $y>0$, 
\begin{equation} \label{non_expl_caseii} \P_y\left(\tau^+(\infty)<\infty\right)=0,\end{equation}
where $\tau^+(\infty)$ has been defined in \eqref{deftauinfty}.
Let $\eps>0$.
Similar computations as for point $i)$ lead to
 \begin{align*}
  \E_{\delta_x}\left[ \sum_{u \in V_t} \left(X_t^u \wedge \zeta \right)^{1-a} \right] 
   & \leq x e^{(\gamma-\gamma') t} + e^{\gamma t}  
   \left( \E_x \left[\mathbf{1}_{\{\tau^-(\eps) <t\}}\left(Y_t \wedge \zeta \right)^{1-a} \right] 
  +  \P_x ( \tau^+(1/\eps) < t )\zeta^{1-a} \right).
 \end{align*}
 
From \eqref{non_expl_caseii}, the last term converges to $0$ when $\eps$ goes to $0$.
Moreover, distinguishing between the cases 
$ \{\tau^-(0)\leq t\}$ and $\{\tau^-(0)> t\} $
and applying the dominated convergence theorem, we obtain that the second term also converges to $0$ when $\eps$ goes to $0$.
This concludes the proof of point {\it ii)}.
\end{proof}

\begin{proof}[Proof of Proposition \ref{prop_constant_diff}]
Let $K>0$.
 Let us first prove that under Assumption \ref{ass_expl_tps_fini} (resp. \ref{ass_abs_tps_fini}),
 \begin{equation}\label{eq:proba=0}  
 \lim_{t \to \infty} \P_{\delta_x}\left( \exists u \in V_t, X_t^u \leq K \right) =0\quad \left(\text{resp. } \lim_{t \to \infty} \P_{\delta_x}\left( \exists u \in V_t, X_t^u \geq 1/K \right) =0 \right).
\end{equation}
Let $t\geq 0$. The first limit is due to the following sequence of inequalities, as $a>1$ under the assumptions of point $i)$:
\begin{align*}
 \P_{\delta_x}\left( \exists u \in V_t, X_t^u \leq K \right)& = \P_{\delta_x}\left( \exists u \in V_t, \left(X_t^u \vee K\right)^{1-a} = K^{1-a} \right)\\
 & \leq \P_{\delta_x}\left( \sum_{u \in V_t} \left(X_t^u \vee K\right)^{1-a} \geq K^{1-a} \right) \leq \E_{\delta_x}\left[ \sum_{u \in V_t} \left(X_t^u \vee K\right)^{1-a}\right]K^{a-1},
\end{align*}
where we used Markov's inequality.
The other case is similar. Then, \eqref{eq:proba=0} follows from Lemma \ref{lem_gene}.
 \end{proof}
\appendix

\section{Proof of Proposition \ref{pro_exi_uni}} \label{proof_ex_uni}

To prove that the SDE \eqref{X_sans_sauts} admits a unique nonnegative strong solution with generator given by $\mathcal{G}$ defined in \eqref{def_gene}, we apply \cite[Proposition 1]{palau2018branching}. The proof is the same as the proof of \cite[Proposition 2.1]{companion}, except that we have to take into account the extra stable term. To prove that we still have a unique nonnegative strong solution with the addition of this term, it is enough to check that for any $n \in \N$, there exists a finite constant $A_n$ such that for any $0 \leq x,y \leq n$,
$$ \int_{\R_+} \left| \mathbf{1}_{\{ x \leq u \}}z \wedge n - \mathbf{1}_{\{ y \leq u \}}z \wedge n \right|\frac{dz}{z^{2+\beta}}du \leq A_n |x-y|,$$
where we recall that $\beta \in (-1,0)$. This is a consequence of the following series of equalities:
\begin{align*}
 \int_{\R_+} \left| \mathbf{1}_{\{ x \leq u \}}z \wedge n - \mathbf{1}_{\{ y \leq u \}}z \wedge n \right|\frac{dz}{z^{2+\beta}}du &= \int_{\R_+} \left| x  - y \right| z \wedge n\frac{dz}{z^{2+\beta}}du\\
 &= |x-y| \left( \int_0^n \frac{dz}{z^{1+\beta}} +n \int_n^\infty \frac{dz}{z^{2+\beta}} \right).
\end{align*}
To prove that it gives the existence and uniqueness of the process at the cell population level, we apply \cite[Theorem 2.1]{marguet2016uniform}.

\section{Proof of Proposition \ref{prop_sol_SDE_auxi}}\label{app:prop_sol_SDE_auxi}

We provide the proof of the proposition in the case $g \neq \beta$. The proof is very similar in the case $g = \beta$ and does not bring new insight. We thus skip it.
The proof is a direct application of \cite[Proposition 1]{palau2018branching}.
Notice that in the statement of \cite[Proposition 1]{palau2018branching}, the functions $b$, $g$ and $h$ do not depend on time, 
unlike the present case of our process. However this additional dependence does not bring any modification to the proofs (which are mostly  derived in the earlier 
paper \cite{li2012strong}).
First according to their conditions (i) to (iv) on page 60, our parameters are admissible.
Second, we need to check that conditions (a), (b) and (c) are fulfilled.

In our case, condition (a) writes as follows: for any $n \in \N$, there exists $A_n<\infty$ such that for any $0 \leq x \leq n$,
$$  \int_0^1 \int_0^\infty \left| (\theta-1)x \mathbf{1}_{\{z \leq f_2(x,s,\theta)\}} \right| dz \kappa(d\theta)\leq A_n(1+x). $$
We have for any $0\leq x \leq n$
\begin{align*}
\int_0^1 \int_0^\infty \left| (\theta-1)x \mathbf{1}_{\{z \leq f_2(x,s,\theta)\}} \right| dz \kappa(d\theta) \leq & 
\int_0^1 2(1-\theta)x(\alpha x+\beta)\kappa(d\theta) \leq n(\alpha n+\beta),
\end{align*}
and thus (a) holds. Now, to satisfy condition (b), it is enough to check that for any $n \in \N$, 
there exists $B_n(t)<\infty$ such that for $0\leq s\leq t$ and
$0 \leq x\leq y\leq n$,
$$ |xf_1(x,s)-yf_1(y,s)|+ \int_0^\infty \int_0^1 (1-\theta) \left| x \mathbf{1}_{\{u \leq f_2(x,s)\}} - y \mathbf{1}_{\{u \leq f_2(y,s)\}} \right| \kappa(d\theta) du
\leq B_n(t)|y-x|. $$
First, we have
\begin{align*}
xf_1(x,s) = gx + F(x)\frac{\alpha \left(e^{(g-\beta)s}-1\right)}{g-\beta+\alpha x\left(e^{(g-\beta)s}-1\right)},
\end{align*}
where $F(x):= 2\sigma^2(x)+\mathbb{E}\left[\mathcal Z ^2\right]p(x).$
For $0\leq x,y\leq n$,
\begin{align*}
&  \left| \frac{F(x)}{g-\beta+\alpha x\left(e^{(g-\beta)s}-1\right)}-
  \frac{F(y)}{g-\beta+\alpha y\left(e^{(g-\beta)s}-1\right)} \right| \\ 
& = \left| \frac{ (F(x)-F(y))(g-\beta)+ 
  \alpha \left(e^{(g-\beta)s}-1\right) (F(x)y-F(y)x)}{\left(g-\beta+\alpha y\left(e^{(g-\beta)s}-1\right)\right)\left(g-\beta+\alpha x\left(e^{(g-\beta)s}-1\right)\right)} \right|\\
&\leq 2\frac{ \left|\sigma(x)-\sigma(y)\right|^2 }{|g-\beta|}+\mathbb{E}\left[\mathcal Z ^2\right]\frac{ \left|p(x)-p(y)\right| }{|g-\beta|}+
  \left|\frac{\alpha\left(e^{(g-\beta)s}-1\right)}{(g-\beta)^2}\right|\left|F(x)y-F(y)x\right|\\
&  \leq 2\frac{ \left|\sigma(x)-\sigma(y)\right|^2 }{|g-\beta|}+\mathbb{E}\left[\mathcal Z ^2\right]\frac{ \left|p(x)-p(y)\right| }{|g-\beta|}+
   \frac{\alpha\left(e^{(g-\beta)t}+1\right)}{(g-\beta)^2}(F(x)|y-x|+\left|F(x)-F(y)\right|x).
\end{align*}
Therefore, using Assumption \ref{ass_A}, there exists $B_{n,1}(t)>0$ such that
\begin{align*}
\left|F(x)\frac{\alpha\left(e^{(g-\beta)s}-1\right)}{g-\beta+\alpha x\left(e^{(g-\beta)s}-1\right)}-
  F(y)\frac{\alpha\left(e^{(g-\beta)s}-1\right)}{g-\beta+\alpha y\left(e^{(g-\beta)s}-1\right)} \right|\leq  B_{n,1}(t)|x-y|.
\end{align*}
To prove that (b) holds, it remains to prove that for any $n \in \N$, there exists $B_{n,2}(t)<\infty$ such that for all
$0 \leq x\leq y\leq n$, and $0\leq s \leq t$,
\begin{align*}
\int_0^1\int_0^\infty  (1-\theta) \left| x \mathbf{1}_{\{u \leq f_2(x,s,\theta)\}} - y \mathbf{1}_{\{u \leq f_2(y,s,\theta)\}} \right| du \kappa(d\theta) 
\leq \frac{1}{2} B_{n,2}(t)|y-x|.
\end{align*}
But for any $|x|,|y|\leq n$, 
$$ \left|x\mathbf{1}_{u \leq f_2(x,s,\theta)}-y\mathbf{1}_{u \leq f_2(y,s,\theta)} \right| \leq
 n \left|\mathbf{1}_{u \leq f_2(x,s,\theta)}-\mathbf{1}_{u \leq f_2(y,s,\theta)} \right|+ |x-y|\mathbf{1}_{u \leq f_2(y,s, \theta)}.$$
Hence 
 \begin{align*} \int_{0}^\infty \left|x\mathbf{1}_{u \leq f_2(x,s,\theta)}-y\mathbf{1}_{u \leq f_2(y,s,\theta)} \right|  du \leq n |f_2(x,s,\theta)-f_2(y,s,\theta)|+ |x-y| f_2(y,s,\theta).
\end{align*}
Next,
\begin{align*}
 & |f_2(x,s,\theta)-f_2(y,s,\theta)|\\
 & =  2\left| (\alpha x+\beta)\frac{g-\beta+\alpha \theta x \left(e^{(g-\beta)s}-1\right)}{g-\beta+\alpha x\left(e^{(g-\beta)s}-1\right)}- (\alpha y+\beta)\frac{g-\beta+\alpha \theta y \left(e^{(g-\beta)s}-1\right)}{g-\beta+\alpha y\left(e^{(g-\beta)s}-1\right)} \right|\\
& \leq 2\alpha |x-y| + 2\alpha\beta\left(e^{|g-\beta|t}-1\right)|x-y|,
\end{align*} 
and
\begin{align*}
 \int_{0}^\infty \left|x\mathbf{1}_{u \leq f_2(x,s,\theta)}-y\mathbf{1}_{u \leq f_2(y,s,\theta)} \right| \leq B_{n,2}(t) |x-y|, 
\end{align*}
where $B_{n,2}(t) =2n\alpha + 2n\alpha\beta\left(e^{|g-\beta|t}-1\right)+2(\alpha n +\beta)$  and condition (b) holds with $B_n(t) = g + B_{n,1}(t)+B_{n,2}(t)/2$. 

It remains to check that (c) is satisfied {\it i.e.} for all $0\leq s\leq t$, $z\geq 0$ and $u\geq 0$ that $x\mapsto x+h(x,s,z,u)$ is nondecreasing, where $h(x,s,z,u) = z\mathbf{1}_{\left\{u\leq f_3(x,s,z)\right\}}$ and that there exists $C_n(t)$ such that for all $0\leq x,y,\leq n$ and $0\leq s \leq t$,
\begin{align*}
|\sigma(x)-\sigma(y)|^2+ \int_0^\infty \int_0^\infty \left(|l(x,y,s,z,u)|\wedge (l(x,y,s,z,u))^2\right) du\pi(dz)\leq C_n(t)|x-y|,
\end{align*}
where $l(x,y,s,z,u) = h(x,s,z,u)-h(y,s,z,u)$. First, notice that for all $0\leq s\leq t$ and $z\geq 0$, $x\mapsto f_3(x,s,z)$ is nondecreasing thanks to Assumption \ref{Ass_E'} so that $x\mapsto x+h(x,s,z,u)$ is nondecreasing. Next, 
\begin{align*}
  \int_0^\infty \left(|l(x,y,s,z,u)|\wedge (l(x,y,s,z,u))^2\right) du
& = \int_0^\infty \left(z\wedge z^2\right)\left|\mathbf{1}_{\left\{u\leq f_3(x,s,z)\right\}}-\mathbf{1}_{\left\{u\leq f_3(y,s,z)\right\}}\right|du\\
& = \left(z\wedge z^2\right)\left| f_3(x,s,z)- f_3(y,s,z)\right|.
\end{align*}
Moreover,
\begin{align*}
& \left| f_3(x,s,z)- f_3(y,s,z)\right|\\
&  \leq |p(x)-p(y)| + \alpha z\left| e^{(g-\beta)s}-1\right|\left|\frac{p(x)}{g-\beta+\alpha x\left(e^{(g-\beta)s}-1\right)} - 
\frac{p(y)}{g-\beta+\alpha y\left(e^{(g-\beta)s}-1\right)}\right|\\
& \leq |p(x)-p(y)| + \alpha z \frac{\left(e^{|g-\beta|t}+1\right)}{|g-\beta|}\left(\left|p(x)-p(y)\right| + 
\alpha  \frac{\left(e^{|g-\beta|t}+1\right)}{|g-\beta|} |yp(x)-xp(y)| \right).
\end{align*}
We get the desired inequality using that $p$ is locally Lipschitz, 
$\int_0^\infty z^2\pi(dz)<\infty$ and that $\sigma$ is $1/2$-H\"olderian.

\section{Lemma \ref{lemtechnq}} \label{App_lemtechnq}

This appendix is dedicated to the statement and proof of a lemma, which is a slightly weaker version of Lemma 4.3 in \cite{BT11}. The only difference
is that they considered a Yule process instead of a birth and death process, and that the finite sets $I$ and $J$ could be arbitrary, whereas we 
impose the condition $J \subset I$.
The statement and proof are deliberately very close to that of Lemma 4.3 in \cite{BT11}. We give the proof in integrality for the sake of readability.

\begin{lemma} \label{lemtechnq}Let $V$ be a denumerable subset and
 $(N_t(i) : t\geq 0)$  be i.i.d. birth and death processes with birth and death rates $r$ and $q< r$ for $i\in V$. Then there
exist $\delta>0$ and a nonnegative nonincreasing function  $G$  on $\R_+$
such that $G(y)\rightarrow 0$ as $y\rightarrow \infty$ and  for all $J\subset I$ finite subsets of $V$ and $x\geq 0$:
\begin{equation*}
\mathfrak{P}(x,\#J,\#I):=\P\Big(\sup_{t\geq 0} \mathbf{1}_{\{\sum_{i\in I} N_t(i)>0\}} \frac{\sum_{i\in J} N_t(i)}{\sum_{i\in I} N_t(i)} \geq x\Big)\leq G\left(\frac{\#I}{\#J}  x\right)+e^{-2\delta\#I}.
\end{equation*}
\end{lemma}
\begin{proof}
From classical results on birth and death processes (see \cite{athreya1972branching}
for instance), we know that for $i \in V$ the process $(N_t(i)e^{-(r-q)t})$ is a non 
negative martingale which converges to a random variable $W$ which is positive on the survival 
event (occurring with probability $(r-q)/r$).
Let us introduce the random variables,
$$M(i):=\sup_{t\geq 0} N_t(i)e^{-(r-q)t}\quad \mbox{ and } \quad m(i)=\inf_{t\geq 0} N_t(i)e^{-(r-q)t}.$$
$(M(i) : i \in V)$ and $(m(i): i \in V)$ are both sequences of finite nonnegative i.i.d.
random variables with finite expectation. 
Moreover, if we introduce, for $i \in V$, the events:
$$ V_\infty(i):= \{ N_t(i) \geq 1, \forall t\geq 0 \} \quad \text{and} \quad 
M_\infty(i):= \{ \exists t<\infty, N_t(i) = 0 \}, $$
and the set
$$ V_\infty:=\{ i \in I, V_\infty(i) \text{ holds}\}, $$
we have that
$ 0=m(i)\leq M(i) $ on the event $M_\infty(i)$, and $0<m(i)\leq M(i)$ on the event $V_\infty(i)$.
As a consequence, for any $\eps \in (0,1)$ and $t \geq 0$, using also that $J\subset I$, we have 
\begin{align*} \mathbf{1}_{\{\sum_{i\in I} N_t(i)>0\}} \frac{\sum_{i\in J} N_t(i)}{\sum_{i\in I} N_t(i)}&\leq  \mathbf{1}_{\{\#V_\infty>\eps \#I\}} \frac{\sum_{i\in J} M(i)}{\sum_{i\in I} m(i)}+\mathbf{1}_{\{\#V_\infty\leq \eps \#I\}} \\
&=  \frac{\#J}{\#I}\left(\mathbf{1}_{\{\#V_\infty>\eps \#I\}}\frac{\sum_{i\in J} M(i)}{\# J} \frac{\# I}{\sum_{i\in I} m(i)}+\mathbf{1}_{\{\#V_\infty\leq \eps \#I\}}\frac{\#I}{\#J}\right).
\end{align*}
Hence, we can bound $\mathfrak{P}$ as follows:
\begin{align}\nonumber\label{eq:app1}
\mathfrak{P}(x,\#J,\#I) \leq  & \P \left( \mathbf{1}_{\{\#V_\infty>\eps \#I\}}\frac{\sum_{i\in J} M(i)}{\# J} \frac{\# I}{\sum_{i\in I} m(i)} \geq \frac{\#I}{\#J}\frac{x}{2} \right)\\
&+ \P \left(\mathbf{1}_{\{\#V_\infty\leq \eps \#I\}}\frac{\#I}{\#J}\geq \frac{\#I}{\#J}\frac{x}{2}\right).
\end{align}

To handle the first term on the right-hand side of \eqref{eq:app1}, we define for $y\geq 0$
$$G(y)=\sup \Big\{  \P\Big(\mathbf{1}_{\{\#V_\infty>\eps \#I\}}\frac{\sum_{i\in J} M(i)}{\# J} \frac{\# I}{\sum_{i\in I} m(i)}\geq y\Big)  : J, I\subset V; \#I < \infty\Big\}.$$
By the law of large numbers, the sequence
$$\mathbf{1}_{\{\#V_\infty>\eps \#I\}}\frac{\sum_{i\in J} M(i)}{\# J} \frac{\# I}{\sum_{i\in I} m(i)}$$
is uniformly tight. So $G(y)\rightarrow  0$ as $y\rightarrow \infty$.

For the second term on the right-hand side of \eqref{eq:app1}, Markov's inequality yields
\begin{align*}
\P \left(\mathbf{1}_{\{\#V_\infty\leq \eps \#I\}}\frac{\#I}{\#J}\geq \frac{\#I}{\#J}\frac{x}{2}\right)\leq \P \left(\#V_\infty\leq \eps\#I\right) \frac{2}{x}.
\end{align*}
To bound the last term, we recall that $\# V_\infty$ is a sum of $\#I$ independent Bernoulli random variables with parameter $1-q/r$.
Hence, for $\eps \leq 1-q/r$, using Hoeffding's inequality, we obtain
\begin{align*}
\P \left(\#V_\infty\leq \eps\#I\right)\leq \exp\left(-2\#I\left(1-\frac{q}{r}-\eps\right)^2\right)
\end{align*}
and it concludes the proof.
\end{proof}

\section{Generalization to a division rate depending on the fragmentation parameter $\theta$}\label{app:generalization SDE}

In some proofs, we need to consider a slight generalization 
of the SDE \eqref{eq:EDS} where an individual with trait $x$ dies 
and transmits a proportion $\theta \in [0,1]$ of its trait to its left offspring
at a rate $r(x)\rtheta(\theta)$, that depends on $\theta$,
where $\rtheta:[0,1]\rightarrow \mathbb{R}_+$ is a nonnegative function. 
However, using the properties of Poisson random measures we can prove that 
a solution to such an SDE can be rewritten 
as the solution to \eqref{eq:EDS} by modifying the death 
rate $r$ and the fragmentation kernel $\kappa$.
\begin{lemma}\label{chgt-eds}
Assume that $\int_0^1 \rtheta(\theta)\kappa(d\theta)<\infty$. Let 
$$\hat{\kappa}(d\theta) = \rtheta(\theta)\left(\int_0^1\rtheta(\theta)\kappa(d\theta)\right)^{-1}\kappa(d\theta),\quad \hat{r}(x) = r(x)\int_0^1\rtheta(\theta)\kappa(d\theta),$$
and $\widetilde{Q}$, $B$ and $N$ be defined as in \eqref{eq:EDS}.
Then, there exists a Poisson random measure $N'$ with intensity $ds\otimes dz\otimes  \hat{\kappa}(d\theta)$ such that $X$ is the pathwise unique solution to 
\begin{align*}
X_t= X_0+ \int_0^t g(X_s)ds+ \int_0^t\sqrt{2 \sigma^2(X_s) }dB_s&+ \int_0^t\int_0^{p(X_{s^-})}\int_{\mathbb{R}_+}z\widetilde{Q}(ds,dx,dz)\\
&+ \int_0^t  \int_0^{\hat{r}(X_{s^-})} \int_0^1 (\theta-1)X_{s^-}N'(ds,dz,d\theta), \nonumber
\end{align*} 
if and only if $X$ is the pathwise unique nonnegative strong solution to
\begin{align*}
X_t= X_0+ \int_0^t g(X_s) ds+ \int_0^t\sqrt{2 \sigma^2(X_s) }dB_s&+ \int_0^t\int_0^{p(X_{s^-})}\int_{\mathbb{R}_+}z\widetilde{Q}(ds,dx,dz)\\
&+ \int_0^t \int_0^1\int_0^{r(X_{s^-})\rtheta(\theta)}(\theta-1)X_{s^-}N(ds,dz,d\theta). \nonumber
\end{align*}
\end{lemma}

\section*{Acknowledgments}
The authors are grateful to V. Bansaye for his advice and comments and to B. Cloez for fruitful discussions.
This work was partially funded by the Chair "Mod\'elisation Math\'ematique et Biodiversit\'e" of VEOLIA-Ecole Polytechnique-MNHN-F.X. and by the French national research agency (ANR) via project MEMIP (ANR-16-CE33-0018).

\bibliographystyle{abbrv}
\bibliography{biblio}

\end{document}